\documentclass[10pt]{article}
\usepackage[T1]{fontenc}
\usepackage[utf8]{inputenc}
\usepackage[english]{babel}
\usepackage{graphicx, color}
\usepackage{amssymb, amsmath, amsfonts}
\usepackage{dsfont}
\usepackage{vmargin}
\usepackage{amsthm} 

\everymath{\displaystyle}

\usepackage{hyperref}
\hypersetup{
	colorlinks=true,
	breaklinks=true,
	urlcolor=blue,
	linkcolor=red,
	bookmarksopen=true,
	pdftitle = $\alpha$-Navier-Stokes
	pdfauthor = Ludovic Gouden\`ege and Luigi Manca,
	pdfsubject = Stochastic $\alpha$-Navier-Stokes
	pdfkeywords = {Navier-Stokes, alpha, stochastic partial differential equations, noise}
}

\numberwithin{equation}{section} 

\newcommand{\R}{\mathbb{R}} 
\newcommand{\N}{\mathbb{N}} 
\newcommand{\E}{\mathbb{E}} 
\newcommand{\sL}{\mathrm{L}}
\newcommand{\sH}{\mathrm{H}}
\newcommand{\sV}{\mathrm{V}}

\newcommand{\dd}{\mathrm{d}} 
\newcommand{\dds}{\mathrm{d}s}
\newcommand{\ddt}{\mathrm{d}t}

\newcommand{\Div}{\mathrm{div\,}} 
\newcommand{\qspace}{\Uset} 
\newcommand{\Prj}{\mathcal{P}} 
\newcommand{\Prb}{\mathbb{P}} 
\newcommand{\Tr}{\mathrm{Tr}} 
\newcommand{\la}{\left\langle}
\newcommand{\ra}{\right\rangle}

\newcommand{\Cp}{C_{\mathcal{P}}} 
\newcommand{\Uset}{\mathcal{U}}
\newcommand{\Vset}{\mathcal{V}}

\newcommand{\rot}{\mathrm{rot}}
\newcommand{\Span}{\mathrm{span}}

\newtheorem{Th}{Theorem}[section]
\newtheorem{Def}[Th]{Definition}
\newtheorem{Coro}[Th]{Corollary}
\newtheorem{Le}[Th]{Lemma}
\newtheorem{Prop}[Th]{Proposition}
\newtheorem{Rem}[Th]{Remark}
\newtheorem{Hyp}[Th]{Hypothesis}

\title{$\alpha$-Navier-Stokes equation perturbed by space-time noise of trace class}

\author{Ludovic Gouden\`ege 
\footnote{CNRS, F\'ed\'eration de Math\'ematiques de CentraleSup\'elec FR 3487, Univ. Paris-Saclay, CentraleSup\'elec, F-91190 Gif-sur-Yvette, France}
\and
Luigi Manca 
\footnote{LAMA, Univ. Gustave Eiffel, UPEM, Univ. Paris Est Creteil, CNRS, F-77447 Marne-la-Vall\'ee, France}
 }

\begin{document}

\maketitle

\begin{abstract}
We consider a stochastic perturbation of the $\alpha$-Navier-Stokes model. The stochastic perturbation is an additive space-time noise of trace class.
Under a natural condition about the trace of operator $Q$ in front of the noise, we prove the existence and uniqueness of strong solution, continuous in time in classical spaces of $\sL^{2}$ functions with estimates of non-linear terms. It is based on a priori estimate of solutions of finite-dimensional systems, and tightness of the approximated solution.

Moreover, by studying the derivative of the solution with respect to the initial data, we can prove exponential moment of the approximated solutions, which is enough to obtain Strong Feller property and irreducibility of the transition semigroup. This leads naturally to the existence and uniqueness of an invariant measure.
\bigskip

 
\end{abstract}

\section{Introduction}

The (stochastic) Navier-Stokes equation has been widely studied under various boundary conditions in domains of $\R^{d}$ with $d=2,3$, also the compressible or incompressible case. But the stochastic version of the $\alpha$-Navier-Stokes model, which has been introduced for modeling of turbulence, has not been studied deeply in the mathematical literature.

The deterministic version has first been introduced by Holm et al. in \cite{ChenFoiasHolmOlsonTitiWynne, HolmMarsdenRatiu:98b} as a Large Eddy Simulation (LES) model (see \cite{Leonard:74}).
Indeed, starting from the fact that a precise description of the fine scales in a turbulent flow may be irrelevant for numerical simulation in engineering applications, the complete Navier-Stokes equation could certainly be relaxed in a weaker form.

From the first study, this relaxation has been derived by applying temporal averaging procedures to Hamilton’s principle for an ideal incompressible fluid flow and Euler-Poincaré variational framework (see \cite{ChenFoiasHolmOlsonTitiWynne}). In \cite{GuermondOdenPrudhomme}, it has been interpreted as a {\em perturbation} of the Leray regularization to restore frame invariance. It is also known as the Lagrangian averaged Navier-Stokes-$\alpha$ (LANS$\alpha$) (see \cite{MarsdenShkoller}) or the viscous Camassa-Holm equation (VCH) (see \cite{BjorlandSchonbek}).


It has been shown that this model possesses a lot of physical properties (conservation laws for energy and momentum), and it is also suitable for numerical simulation. 
For instance, these equations on a fluid $v$ possess a Kelvin-Noether circulation theorem and conserves helicity (see \cite{FoiasHolmTiti:01}) which is defined on a volume $V$ by: 
\[
\mathcal{H}(V)(t)=\int_V v(t,x)\cdot (\nabla\times v(t,x))\ dx^{3}. 
\]
This $\alpha$-Navier-Stokes model must not be confused with Leray-$\alpha$ regularization of the Navier-Stokes, or a hyper-viscous Navier-Stokes equation (bi-Laplacian) since these last models do not provide the same physical properties.
We can also cite the results about higher-order Leray models and deconvolution (see \cite{LaytonLewandowski, Rebholz}).



In \cite{ChenHolmMargolinZhang}, Direct Numerical Simulations (DNS) have been realized to demonstrate that it 
reproduces most of the large scale features of Navier-Stokes turbulence (detailed previously) even when these simulations do not resolve the fine scale dynamics, at least in the case of turbulence in a periodic box.
They have compared vorticity structures and alignment, and also two point statistics (e.g. speed increments and flatness) to illustrate the altered dynamics of the alpha models.\bigskip

The original Navier-Stokes equation reads\footnote{Notice that $(v\cdot \nabla v)_i = \sum_j v_j\partial_j v_i$ for $i=1,...,d$ with $d=2,3$.}
\[
   \frac{\partial v}{\partial t}+v\cdot\nabla v +\nabla p = \nu \Delta v +f
\]
for a pressure $p$, a forcing $f$, a constant kinematic viscosity $\nu$ and a fluid $v$ of constant density. The $\alpha$-Navier-Stokes model for an incompressible fluid $v$ is given by
\[
\begin{cases}
\dfrac{\partial v}{\partial t}+\bar v\cdot\nabla v +(\nabla \bar v)^T\cdot v+\nabla p = \nu \Delta v +f,\\
\Div \bar v=0\\
\bar v = \left(I-\alpha^2\Delta\right)^{-1}v,
\end{cases}
\]
where we consider a periodic fluid in a box, or a fluid with homogeneous Dirichlet condition, and the initial data is given by $v|_{t=0} = v_{0}$.

Alternatively, using the Helmoltz operator $I-\alpha^2\Delta$, and eliminating $v$ by setting $v=\left(I-\alpha^2\Delta\right)\bar v$, we can write
\[
\begin{cases}
\frac{\partial (\bar v-\alpha^2 \Delta \bar v)}{\partial t}+ \bar v\cdot\nabla (\bar v-\alpha^2 \Delta \bar v) +(\nabla \bar v)^T\cdot (\bar v-\alpha^2 \Delta \bar v)+\nabla p = \nu \Delta (\bar v-\alpha^2 \Delta \bar v) +f,\\
\Div \bar v=0.
\end{cases}
\]
Sometimes the term $(\nabla \bar v)^T\cdot \bar v$ does not appear in model, since it should disappear with Leray projection. We can also use a dynamic pressure $\tilde{p}$ to obtain the $\alpha$-Navier-Stokes model in rotational form
\[
\begin{cases} 
\frac{\partial v}{\partial t}-\bar v\times( \nabla\times v)+\nabla \tilde{p}=\nu\Delta v + f,\\
\Div \bar v=0,
\end{cases}
\]
and the two equations on the vorticity $ q:=\nabla \times v$ and the helicity $H:= \langle v , \nabla \times v\rangle = \langle v , q\rangle $ are given by
\[
\begin{cases} 
\frac{\partial q}{\partial t}+\bar v \cdot \nabla q -q \cdot \nabla \bar v =\nu\Delta q + \nabla\times f,\\
\\
\frac12\frac{d H(t)}{d t} =\int_{V} \nu\Delta v \cdot q \  dx^{3}+\int_{V} (\nabla\times f) \cdot (\nabla \times v) \ dx^{3}.\\
\end{cases}
\]


In this paper, we will consider the stochastic version of $\alpha$-Navier-Stokes equation by substituting to the forcing term $f$, a time derivative of a gaussian Wiener process, to take into account random agitations or uncertainties.

In the next sections, we will clarify the definition of an abstract setting such that the stochastic $\alpha$-Navier-Stokes equation reads
\begin{equation}
   \begin{cases} 
    du+\left(\nu Au+(I+\alpha^2A)^{-1}\widetilde{B}(u,u+\alpha^2A u)\right)dt = d\xi & \text{in } \sH,\\
    u(0)=u_0\in \sH,
  \end{cases}
\end{equation}
where $\xi$ is a noise term, and 
with the property that if $\alpha=0$, this is equivalent to stochastic Navier-Stokes
\begin{equation}
   \begin{cases} 
    du+\left(\nu Au+B(u,u)\right)dt =d\xi & \text{in } \sH,\\
    u(0)=u_0\in \sH.
  \end{cases}
\end{equation}
For the stochastic Navier-Stokes equation, there exist many results since the early work of Bensoussan and Temam \cite{BensoussanTemam}. It is well-known that there exists a probabilistic weak (martingale) solution in the three-dimensional case (see for instance \cite{FlandoliGatarek, MikuleviciusRozovskii}). Concerning the uniqueness, there exists a (probabilistic) strong maximal local solution in $\mathrm{W}_{p}^{1}$ with $p>3$ obtained in \cite{BrzezniakPeszat, MikuleviciusRozovskii}. In \cite{GlattHoltzZiane, Mikulevicius}, there is existence and uniqueness of local (probabilistic) strong pathwise solution in $\mathrm{W}^{1}_{2}=\sH^{1}$. Using a Galerkin approximation and Kolmogorov equation, the authors in \cite{DaPratoDebussche} are able to construct a transition semi-group with a unique invariant measure, which is ergodic and strongly mixing.
For the two-dimensional case, there exist global results in $\sL^{2}$ space. It worths mentioning the paper \cite{CaraballoMarquezDuranReal} where the authors have studied the asymptotic behavior in a deterministic context but with some terms containing some kind of memory (e.g. delay). They prove existence, uniqueness, and exponential convergence to a stationary solution provided the viscosity is large enough.

Concerning the $\alpha$-Navier-Stokes equation, the authors of \cite{CaraballoRealTaniguchi} have proved the existence and uniqueness of variational solutions with multiplicative noise. But they have only obtained estimation of the fourth moments in the random space $\sL^{4}(\Omega)$.
In \cite{DeugoueSango}, the authors have relaxed the Lipschitz condition (but conserving the sublinearity) on the coefficients in the second member and in the multiplicative noise, to obtain the existence of weak solutions. The uniqueness is satisfied by assuming some Lipschitz conditions. They have estimations for all the moments of the solution, but the noise only appears in a finite number of modes.

In our paper, we are able to show that there exists a unique strong solution that has moments of any order for an additive infinite-dimensional space-time noise.
With an additional assumption on the noise, we show a Strong Feller-type property of the transition semigroup associated to the solution of the SPDE.
This leads naturally to the existence and uniqueness of invariant measure.
Moreover, we are able to prove exponential moments of the approximated solutions which is enough to obtain a concentration property for the invariant measure.

We will make the assumption that the noise process is of trace class (more details will be given later in Section \ref{sec:preliminary}). In Section \ref{sec:preliminary} and \ref{sec:operatorB}, we will use the classical abstract formalism of stochastic PDEs to obtain a strong and weak solution with continuous path in Hilbert spaces of square integrable functions with free divergence (see Section \ref{sec:def} for definition of solutions). Moreover, and this is the crucial point of this article, we will give very strong estimates of moment of the solution in Section \ref{sec:results}. We have clarified the dependency in the parameter $\alpha$, such that in the limit $\alpha\rightarrow 0$, we will recover the expected behavior of blow-up of the moment of classical Navier-Stokes solution.

The proof relies on the existence and uniqueness of solution for Galerkin approximated problems (Section \ref{sec:approximated}), which possesses exponential moments (Section \ref{sec:expo}). We have proved estimates in Sobolev spaces (Section \ref{sec:Sobolev}), using compactness argument (Section \ref{sec:compactness}) to make identification of the limit (Section \ref{sec:proof}).
The results in Section \ref{sec:expo} have been used to derive a priori estimates and some concentration properties for the invariant measure in Section \ref{sec:invariant}.

\section{Preliminaries and abstract formulation}\label{sec:preliminary}

Depending of the boundary conditions (periodic or homogeneous Dirichlet) we have to define spaces of free divergence vector fields to treat with classical operators of the $\alpha$-Navier-Stokes equations.
{\em In case of periodic boundary conditions}, we consider a domain $\Uset = [0,L]^{d}$ with $d=2,3$ and we set 
\[
\text{ (periodic) } \qquad \Vset=\{ \varphi\in (\mathrm{P}_{trig})^{d}: \Div \varphi=0 \text{ in } \Uset
 \text{ and } \int_{\Uset} \varphi(x) dx = 0
 \}
\]
 the classical space of vector valued trigonometric polynomial functions with free divergence and with vanishing mean. {\em In case of homogeneous Dirichlet boundary condition}, we consider a smooth domain $\Uset \subset \R^{d}$ with $d=2,3$ which is bounded, open and simply connected. 
We set 
\[
\text{ (Dirichlet) } \qquad\Vset=\{ \varphi\in (C_0^\infty(\Uset))^d: \Div \varphi=0 \text{ in } \Uset\}
\]
 the classical space of infinitely differentiable functions with free divergence and with compact support in $\Uset$. Then the spaces $\sH$ and $\sV$ are the closure of $\Vset$ in $(\sL^2(\Uset))^d$ and in $(\sH^1(\Uset))^d$ respectively.
We denote by $|\cdot|_2$, $|\cdot|_\sV$ the norms in $\sH$, $\sV$ and by $\la \cdot,\cdot\ra$ the standard scalar product in $\sH$.

\begin{Rem}
Let $\vec{n}$ denote the outward normal to $\partial \Uset$, then, following \cite{Temam}, we can characterize the space $\sH$ in the homogeneous Dirichlet case as
\[
\sH = \{ u \in (\sL^2(\Uset))^d : \Div u = 0 \text{ in }  \Uset \text{ and }\vec{n}\cdot u = 0 \text{ on } \partial\Uset\}.
\]
Its orthogonal complement in $(L^2(\Uset))^d$ is
\[
\sH^{\bot} = \{ u \in (\sL^2(\Uset))^d :  u = \nabla p  \text{ in } \Uset \text{ and } p \in \sH^{1}(\Uset)\},
\]
and the space $\sV$ is
\[
\sV = \{u \in (\sH^{1}_{0}(\Uset))^d : \Div u = 0  \text{ in } \Uset\}.
\]
\end{Rem}

We denote by $\Prj : (\sL^2(\Uset))^{d}\to \sH$ the usual orthogonal Leray projector and by $A$ the Stokes operator
\[
   A:=-\Prj\Delta : D(A)\to \sH,
\]
with domain $D(A)=(\sH^2(\Uset))^d\cap \sV$. 
The operator $A$ is a positive self adjoint operator.
Its inverse, $A^{-1}:\sH\to \sH$, is a compact self adjoint operator, thus $\sH$ admits an orthonormal basis $\{e_{k}\}_{k\in\N^{*}}$ formed by the eigenfunctions of $A$, i.e. $A e_{k} = \lambda_{k}e_{k}$, with $0< \lambda_{1}\leq \lambda_{2}\leq \cdots \leq \lambda_{k} \rightarrow+\infty$. For $\rho\in\R$, the Sobolev spaces $D(A^\rho)$ are the closure of $C_0^\infty(\qspace)$ with respect to the norm
\[
    \|x\|_{D(A^\rho)}=\left(\sum_{k\in\N^{*}}\lambda_k^{2\rho}\la x,e_k\ra^2  \right)^{\frac12}.
\]
As well known (see, for instance, \cite{FoiasHolmTiti:02}) the operator $A$ can be continuously extended to $\sV=D(A^{\frac12})$ with values in $\sV'=D(A^{-\frac12})$ such that for all $u$, $v \in \sV$
\[
\la Au,v\ra_{\sV',\sV}=(A^{1/2}u,A^{1/2}v)=\int_\qspace (\nabla u\cdot \nabla v)\ dx,
\]
Similarly $A^2$ can be continuously extended to $D(A)$ with values in $D(A)'$ (the dual space of the Hilbert space $D(A)$) such that
for all $u$, $v \in D(A)$
\[
\la A^2u,v\ra_{D(A)',D(A)}=(Au,Av).
\]
One can show that there is a constant $c>0$ such that for all $w \in D(A)$
\[
c^{-1} |Aw|_{2} \leq \|w \|_{\sH^{2}} \leq c\  |Aw|_{2}.
\]
We assume that $\xi$ is the time derivative of a gaussian noise of the form $QW(t)$, $t\geq0$ where 
\begin{itemize}
 \item $K$ is a separable Hilbert space (norm $|\cdot|_{K}$ and scalar product $\la\cdot,\cdot\ra_{K}$);
 \item $W(t),t\geq0$ is a cylindrical white noise defined on a stochastic basis $(\Omega, \mathcal{F},(\mathcal{F}_t)_{t\geq0}, \Prb)$ with values in $K$;
 \item $Q:K\to \sH$ is a Hilbert-Schmidt operator.
\end{itemize}
%
%
The equation takes the abstract form
\begin{equation}\label{eq.aNSabstract}
   \begin{cases} 
    du+\left(\nu Au+(I+\alpha^2A)^{-1}\widetilde{B}(u,u+\alpha^2A u)\right)dt =QdW(t) & \text{in } \sH,\\
    u(0)=u_0\in \sH.
  \end{cases}
\end{equation}

\begin{Hyp} \label{hyp.trace}
We assume that the operator $Q$ and $A$ satisfies the following trace class condition
\[
      \Tr[Q^*(I+A)Q]<\infty.
\]
\end{Hyp}
The next assumption will be useful when we shall study the uniqueness of an invariant measure for the semigroup associated to the solution of the SPDE. 
Essentially, it allow to use the Bismut-Elworthy formula and derive a Strong Feller-type property.
\begin{Hyp} \label{hyp.trace2}
The operator $Q:K\to \sH$ is invertible and  $D(A^{3/2})\subset D(Q^{-1})$.
\end{Hyp}

\begin{Rem} \label{rem.QAbounded}
For two separable Hilbert spaces $X,Y$, let us denote by $\mathcal{L}_2(X;Y)$ the set of Hilbert-Schmidt operators $B:X\to Y$. 
The assumption in Hypothesis \ref{hyp.trace} means that $Q\in \mathcal{L}_2(K;\sV)$.
  %
 Moreover, for any $x\in \sV$
  \[ 
    |Q^*A^{\frac12}x|^2\leq |x|_2^2\Tr[Q^*AQ].
  \]
  Then, $Q^*A^{\frac12}:\sV\to K$ can be extended to a bounded operator on $\sH$ (we still denote it by $Q^*A^{\frac12}$). 
  Similarly, for $x\in D(A)$
 \[
     |Q^*Ax|^2\leq |A^{\frac12}x|_2^2\Tr[Q^*AQ]=|\nabla x|_2^2\Tr[Q^*AQ]
 \]
 therefore $Q^*A$ can be extended to a bounded linear operator on $\sV$.

 With this in mind, we get the following formula which will be useful in the following
\begin{equation} \label{eq.QAbounded}
   |Q^*(I+\alpha^2A)x|^2 =|Q^*(I+\alpha^2A)^{\frac12}(I+\alpha^2A)^{\frac12}x|^2 \leq \Tr[Q^*(I+\alpha^2A)Q]\left(|x|_2^2+\alpha^2|\nabla x|_2^2 \right).
\end{equation}
\end{Rem}

\section{The nonlinear operator}\label{sec:operatorB}
Some results of this section can be found in \cite{FoiasHolmTiti:02}. We recall it here for completeness. First remark that, following classical description of the Navier-Stokes equation, for $u,v\in \sV$ we usually use the bilinear operator
\[
B(u,v):=\Prj\left((u\cdot\nabla)v\right)=(\nabla v)^Tu \quad\text{ such that }\quad (B(u,v))^j=\sum_{i=1}^{d}u^i\frac{\partial v^j}{\partial x^i}, \text{ for } j=1, \dots, d.
\]
If $u,v,w\in \sV$ then we can observe that
\[
   \la B(u,v),w\ra = -\la B(u,w),v\ra.
\]
For the $\alpha$-Navier-Stokes equation, recalling the vector identity
\[
u \times (\nabla \times v) :=  u\times \rot\  v = \left(\nabla v-(\nabla v)^T  \right)u
\]
we can set -by analogy with Navier-Stokes equation- a bilinear operator
\[
  \widetilde{B}(u,v):= -\Prj\left(u \times (\nabla \times v)\right).
\]
Following again the classical framework, we set $b(u,v,w)$ the trilinear operator
\[
  b(u,v,w)=\sum_{i,j=1}^{d}\la u^i\frac{\partial v^j}{\partial x^i},w^j \ra= \la (u\cdot \nabla)v,w\ra 
\]
then for $u,v,w\in \sV$
\[
    b(u,v,w)= \la (\nabla v)^Tu,w\ra =\la u, (\nabla v) w\ra.
\]
It is easy to show that for $u,v,w\in \sV$ we have
\[
   b(u,v,w)=-b(u,w,v) \text{ and } b(u,v,v)=0.
\]
This implies that for any $w\in \sV$ (actually, $w\in (L^2(\Uset))^d$ is enough)
\[
  \la \widetilde{B}(u,u),w\ra = b(u,u,w)-b(w,u,u)=b(u,u,w)=\la (u\cdot\nabla)u,w\ra= \la B(u,u),w\ra
\]
and so for $\alpha=0$ equation \eqref{eq.aNSabstract} becomes the Navier-Stokes equation.
We obtain the following identity for $u,v,w\in \sV$, since $w=\Prj w$,
\begin{eqnarray}
b(u,v,w)-b(w,v,u) &=& \la (\nabla v)^Tu,w\ra-\la (\nabla v)^Tw,u\ra = \la (\nabla v)^Tu,w\ra-\la w,(\nabla v) u\ra\\
&=& -\la \left((\nabla v) u- (\nabla v)^Tu\right),\Prj w\ra =  \la \widetilde{B}(u,v), w\ra\\
\Big(&=& \la (\nabla u)^Tw- (\nabla w)^Tu,v\ra \Big)
\end{eqnarray}


\begin{Prop} \label{prop.1.1}
We have the following estimations which will be used later in the proofs.\\
\noindent 
(i) $\widetilde{B}$ can be extended continuously to $\sV\times \sV$ with values in $\sV'$; for any $u,v,w\in \sV$ it satisfies
\begin{eqnarray*}
 &\left|\la \widetilde{B}(u,v),w\ra_{\sV',\sV}\right|\leq c|u|_2^{1/2}|u|_\sV^{1/2}|v|_\sV|w|_\sV,\\
 &\left|\la \widetilde{B}(u,v),w\ra_{\sV',\sV}\right|\leq c|u|_\sV|v|_\sV|w|_2^{1/2}|w|_\sV^{1/2}.\\
\end{eqnarray*}
\noindent
(ii) $\widetilde{B}$ satisfies :
\[
  \la \widetilde{B}(u,v),w\ra_{\sV',\sV}=-\la \widetilde{B}(w,v),u\ra_{\sV',\sV},\quad  \forall u,v,w\in \sV;
\]
\[
   \la \widetilde{B}(u,v),u\ra_{\sV',\sV}=0,\quad \forall u,v\in \sV.
\]
\[
   \la \widetilde{B}(u,v),v\ra_{\sV',\sV}=-b(v,v,u),\quad \forall u,v\in \sV.
\]

\noindent
(iii) 
\[
   \left|\la \widetilde{B}(u,v),w\ra_{D(A)',D(A)}\right| \leq c |u|_2 |v|_\sV |w|_\sV^{1/2}|Aw|_2^{1/2},\quad \forall u\in \sH, v\in \sV, w\in D(A).
\]
and 
\[
   \left|\la \widetilde{B}(u,v),w\ra\right|\leq c |u|^{1/2}_\sV |Au|^{1/2}_2|v|_\sV |w|_2,\quad \forall u\in D(A), v\in \sV, w\in \sH.
\]
(iv) For any $u\in \sV$, $v\in \sH$, $w\in D(A)$ it holds
\[
  \left|\la \widetilde{B}(u,v),w\ra_{D(A)',D(A)}\right|\leq c\left(   |u|^{1/2}_2 |u|^{1/2}_\sV|v|_2 |Aw|_2 + |u|_\sV|v|_2 |w|_\sV^{1/2}|Aw|_2^{1/2}   \right).
\]
\noindent
(v) For any $u\in D(A)$, $v\in \sH$, $w\in \sV$  it holds
\[
  \left|\la \widetilde{B}(u,v),w\ra_{\sV',\sV}\right| \leq c \left(   |u|^{1/2}_\sV |Au|^{1/2}_2|v|_2 |w|_\sV + |Au|_2|v|_2 |w|_2^{1/2}|w|_\sV^{1/2}   \right).
\]
(vi) For any $u\in \sV$, $v\in \sH$, $w\in D(A)$ it holds
\[
      \left|\la \widetilde{B}(u,v),w\ra_{D(A)',D(A)}\right|\leq c |u|_\sV|v|_2 |Aw|_2  .
\]
\end{Prop}
\begin{proof}
Points {\em (i)--(v)} can be found in  \cite{FoiasHolmTiti:02}.
By the results obtained in  \cite{FoiasHolmTiti:02}, for some $c>0$ it holds $c|Aw|_2\leq \|w\|_{\sH^2}\leq c^{-1}|Aw|_2$ for any $w\in D(A)$ 
and $c|A^{1/2}w|_2\leq \|w\|_{\sH^1}\leq c^{-1}|A^{1/2}w|_2$ for any $w\in \sV$.
Then, since $D(A^{1/2})=\sV$, $|w|_2\leq \|w\|_{\sH^1}\leq c|w|_\sV$ and $|w|_\sV\leq c|Aw|_\sV$ for some $c>0$ independent by  $w$.
Consequently, {\em (vi)} of Proposition \ref{prop.1.1} follows by {\em (iv)}.
\end{proof}

\begin{Rem} In \cite{LiLiuXie}, there is a study of a Leray-$\alpha$ model with fractional power of Laplace operator with periodic boundary conditions.  The abstract framework is very similar, but the assumption on the noise states that, for some sufficiently large index $N$, $Range(Q) = P_{N}\sH$ where $P_{N}$ is the projector on the finite dimensional space $Vect(e_{1}, \dots, e_{N})$. In this case, $Q$ vanishes on the complement space, such that their results only apply to finite dimensional noises (highly degenerate), but with enough noise to ensure hypoellipticity (also they need large enough viscosity).
\end{Rem}

\section{Definition of solution}\label{sec:def}
%

%
%
%
We are now able to define the concept of solution of equation  \eqref{eq.aNSabstract}.
\begin{Def}[Strong solution] \label{def.sol.strong}
 Assume that the linear operators $Q$ satisfy Hypothesis \ref{hyp.trace} and let $W(t), t\geq0$ be a cylindrical Wiener process with values in $\sH$.
 Also assume that $u_0$ is a random variable with values in $\sH$, independent by the filtration generated by $W(t),t\geq0$.
We say that a stochastic process $u(t),t\geq 0$ with values in $\sH$ is a strong solution of \eqref{eq.aNSabstract} starting by $u_0$ if
\begin{itemize}
\item  $u(t)$ has paths in $C([0,\infty[;\sH)$ and it is adapted to the filtration generated by $W(t),t\geq0$ ;
\item  for any $T>0$ and $\Prb$-almost surely, 
 \begin{equation} \label{eq.def.sol0}
    \int_0^T|Au(t)|_2\ddt +\int_0^T | (I+\alpha^2 A)^{-1}\widetilde{B}(u(t), u(t)+\alpha^2 Au(t))|_2\ddt +\Tr[Q^*Q]<\infty;
 \end{equation}
\item for any $t\geq0$ the process $u(\cdot)$ verifies 
\begin{equation}   \label{eq.def.sol}
   u(t)+\nu\int_0^tAu(s)\dd s+ \int_0^t (I+\alpha^2A)^{-1}\widetilde{B}(u(s),u(s)+\alpha^2Au(s))\dd s=  u_0+QW(t).
 \end{equation}
\end{itemize}
\end{Def}

\begin{Def}[Weak solution] \label{def.sol.weak} Let $T>0$ and let $\mu_{0}$ be a probability measure on $(\sH, \mathcal{B}(\sH))$. 
Assume that the linear operator $Q$ satisfies Hypothesis \ref{hyp.trace}. 
We say that  $(u, (\Omega, \mathcal{F}, \mathbb{P}, (\mathcal{F}_{t})_{t\in[0,T]}),W)$ is a weak solution of \eqref{eq.aNSabstract} with initial distribution $\mu_{0}$ if 
\begin{itemize}
\item  $(\Omega, \mathcal{F}, \mathbb{P}, (\mathcal{F}_{t})_{t\in[0,T]})$ is a complete filtered probability space;
\item $u$ has paths in $C([0,\infty[;\sH)$  and is adapted to the filtration $(\mathcal{F}_{t})_{t\in[0,T]}$;
\item the law of $u(0)$ is $\mu_{0}$;
\end{itemize}
Moreover,  $\Prb$-a.s., \eqref{eq.def.sol0} and \eqref{eq.def.sol} hold.
\end{Def}

\section{Main result}\label{sec:results}
\begin{Th}\label{thm.intro}
Assume that the linear operator $Q$ satisfies Hypothesis \ref{hyp.trace}.
Let $u_0$ be a random variable with values on $(\sV,\mathcal{B}(\sV))$ such that for some $k\geq1$ it holds
\[
    \E\left[\left(|u_0|_2^2+|\nabla u_0|_2^2\right)^k  \right]<\infty. 
\]
Then there exists a unique strong solution $u(t), t\geq0  $ 
of problem \eqref{eq.aNSabstract} with initial value $u(0)=u_0$. 
Moreover,   there exists a constant $c>0$, depending only on $k,Q,\qspace$ such that for any $T>0$ it holds
\begin{equation} \label{eq.aNSCHestimate}
 \begin{split}
  &\E\left[\sup_{0\leq t\leq T} \left(|u(t)|_{2}^{2}      
      \right)^k \right] 
   \leq   c\left(\E\left[\left(|u_0|_2^2+\alpha^2|\nabla u_0|_2^2\right)^k  \right]+T\right)\\
  &\E\left[\int_0^{T} \left( |u(s)|_2^2 +\alpha^2 |\nabla u(s)|_2^2\right)^{k-1} \left(|\nabla u(s)|_2^2 +\alpha^2|A  u(s)|_2^2 \right) \dd s\right]
    \leq c\left(\E\left[\left(|u_0|_2^2+\alpha^2|\nabla u_0|_2^2\right)^k  \right]+T\right)
   \end{split}
\end{equation}
\end{Th}

We denote by $B_b(\sV)$ the set of real Borel measurable and bounded functions $\varphi:\sV\to \R$.
The transition semigroup $P_t,t\geq 0$  associated to the solution of problem \eqref{eq.aNSabstract} is defined by
\[
   P_t\varphi(x)=\E\left[\varphi(u(t,x))\right]
\]
where $\varphi\in B_b(\sV)$ and $u(t,x)$ is the solution of \eqref{eq.aNSabstract} at time $t$ starting at point $x\in \sV$.

\begin{Th} \label{thm.invmeas}
 Le us assume that Hypothesis \eqref{hyp.trace}, \eqref{hyp.trace2} hold.
 Then, there exists a unique invariant probability measure $\mu$ for the transition semigroup $P_t$, $t\geq0$.
 Moreover $\mu(D(A))=1$ and for any $\varepsilon>0 $ sastisfying $-\nu+2\varepsilon\lambda_1^{-1}\Tr[Q^*(I+\alpha^2A)Q]< 0$ there exists $K_\varepsilon>0$ such that
 \begin{equation} \label{eq.muexp}
  \int_{\sV} e^{\varepsilon\left(|x|_2^2+\alpha^2|\nabla x|_2^2\right)}\left(|\nabla x|_2^2+\alpha^2|A x|_2^2 \right) \mu(dx) 
  \leq \frac{K_\varepsilon}{ \varepsilon(\nu-\varepsilon\lambda_1^{-1}\Tr[Q^*(I+\alpha^2A)Q])}.
 \end{equation}
\end{Th}

%
%
%
%
%
\section{The approximated problem}\label{sec:approximated}

Here we use the Galerkin approximation method. Recall that we have denoted by $0<\lambda_1<\lambda_2<\ldots$ the eigenvalues of the Stokes operator $A$ in its domain $D(A)$ and by $e_1,e_2,\ldots$ the correspondent eigenvectors. 
Then, $P_n:\sH\to \sH$ is the orthogonal projection of $\sH$ onto $ \Span \{e_1,\ldots, e_n\}$.
We set $\tilde  B_n(u,v):= P_n\widetilde{B}(P_n u,P_nv)$. Clearly, the results given in Proposition \ref{prop.1.1} remain valid for $\widetilde{B}_n$.
The Galerkin approximated problem is given by the equation
\begin{equation}  \label{eq.approx}
    \begin{cases} 
    du_n+\left(\nu Au_n+(I+\alpha^2A)^{-1}\widetilde{B}_n(u_n,u_n+\alpha^2Au_n)\right)dt
    =P_nQdW(t) & \text{in } \sH,\\
    u(0)=P_n u_0. 
  \end{cases}
\end{equation}
Actually, \eqref{eq.approx} is a finite dimension ordinary stochastic differential equation with a polynomial nonlinearity.
Then, there exists a unique local solution up to a (possible) blow up time $\tau_n>0$.
The results of this section concern the existence and uniqueness of a global solution for the approximated equations (i.e., $\tau_n=\infty$ almost surely). 
Moreover, we shall show uniform estimates on the solutions $u^n$, which will be essential in order to use compactness arguments.

\begin{Th} \label{thm.moments}
 Let $u_0\in \sV$ and assume that Hypothesis \ref{hyp.trace} holds.
 Then, for any $n\in \N$, $T>0$ there exists a solution $u_n\in \sL^2([0,T];D(A))$ of problem \eqref{eq.approx}.
 Moreover, for any   $k\in \N^*$ there exists a constant $c=c(k,\qspace,Q)>0$ such that for any $n\in \N^*$, $t\geq 0$
 \begin{multline*}
    \E\left[\left( |u_n(t)|_2^2 +\alpha^2 |\nabla u_n(t)|_2^2\right)^k \right] \\ 
    +k\nu \E\left[\int_0^{t} \left( |u_n|_2^2 +\alpha^2 |\nabla u_n|_2^2\right)^{k-1} \left( |\nabla u_n|_2^2       +\alpha^2|A  u_n|_2^2 \right) \dd s\right]
     \leq \E\left[\left(|u_0|_2^2 +\alpha^2 |\nabla u_0|_2^2\right)^k\right] +ct.
 \end{multline*}
%
%
%
\end{Th}
%
%
\begin{proof}
Set 
\begin{equation}  \label{eq.defF}
   \mathcal{F}(t)=\mathcal{F}(t,u_n)=  |u_n|_2^2 +\alpha^2 |\nabla u_n|_2^2
\end{equation}
For any $N>0$, $n\in \N^*$ we consider the stopping time
\begin{equation} \label{eq.stopping}
   \tau_N^n=\begin{cases}  
              \inf\{t:\mathcal{F}(t,u_n(t))>N\} &\text{ if the set is non empty }\\
              +\infty &\text{ otherwise.}
            \end{cases}
\end{equation}
Notice that \eqref{eq.approx} is a system of ordinary differential equations with polynomial nonlinearities.
Then, there exists a local solution $u_n$ up to a blow up time $\tau(\omega)$. 
Since the function $u_n(t\wedge \tau_N^n)$ is bounded by $N$, we can apply the  Itô formula  to obtain 
\begin{eqnarray*}
  \mathcal{F}^k(t \wedge \tau_N^n) &+& 2k\nu\int_0^{t \wedge \tau_N^n}\mathcal{F}^{k-1} \times  \left(|\nabla u_n|_2^2       +\alpha^2|A  u_n|_2^2  \right) \dd s \\
     &=& \mathcal{F}^k(0)+ k\int_0^{t \wedge \tau_N^n}\mathcal{F}^{k-1} \Tr[(P_n Q)^*(I+\alpha^2A)(P_n Q)]\dd s
             \\
    && + 2k(k-1) \int_0^{t \wedge \tau_N^n}\mathcal{F}^{k-2}  \left( |(P_n Q)^*(I+\alpha^2A)u_n|_2^2 \right)\dd s\\
     &&  + 2k\int_0^{t \wedge \tau_N^n}\mathcal{F}^{k-1}\la u_n+\alpha^2Au_n,(P_n Q)\dd W(s)\ra 
          \\
     &=&  \mathcal{F}^k(0)+I_1+I_2+M_t
\end{eqnarray*}
where $M_t$ is the martingale term.
  Taking into account \eqref{eq.QAbounded} there  exists  $c_1>0$, independent by $u_n$ and $T$, such  that
\[
   I_1+I_2\leq c_1\int_0^{t \wedge \tau_N^n}\left( \mathcal{F}^{k-1}+\mathcal{F}^{k-2}(|u_n|_2^2+\alpha^2|\nabla u_n|_2^2)\right)\dds \leq 2c_1\int_0^{t \wedge \tau_N^n} \mathcal{F}^{k-1}\dds.
\]

By Poincaré inequality $|x|_2\leq \Cp |\nabla x|_2$ and by the definition of the operator $A$, $|\nabla x|_2\leq  \frac{1}{\sqrt{\lambda_1}}|Ax|_2$.
Then, $\mathcal{F}^{k}\leq c_2\mathcal{F}^{k-1}(|\nabla x|_2^2+\alpha^2|Ax|_2^2)$, where $c_2>0$ depends only on $\qspace$.
By Young's inequality there exists $c_3>0$, depending only by $Q$, $\qspace$, $k$, such that  
$\mathcal{F}^{k-1}\leq k\nu\mathcal{F}^{k}/(2c_1c_2)+c_3/(2c_1)$. Then
 \begin{eqnarray*}
        2c_1\int_0^{t \wedge \tau_N^n} \mathcal{F}^{k-1}\dds 
   &\leq&   \int_0^{t \wedge \tau_N^n}\left(\frac{k\nu}{c_2} \mathcal{F}^{k}+c_3\right)\dd s\\
   &\leq& k\nu\int_0^{t \wedge \tau_N^n}\mathcal{F}^{k-1}(|\nabla u_n|_2^2+\alpha^2|A u_n|_2^2)\dd s+c_3t \wedge \tau_N^n.
 \end{eqnarray*}
We get 
 \begin{multline}\label{eq.boundsF}
    \mathcal{F}^k(t \wedge \tau_N^n) + k\nu\int_0^{t \wedge \tau_N^n}\mathcal{F}^{k-1} \times  \left(|\nabla u_n|_2^2       +\alpha^2|A  u_n|_2^2  \right) \dd s \\
    \leq \mathcal{F}^k(0) 
					   +c_3t \wedge \tau_N^n+
 2k\int_0^{t \wedge \tau_N^n}\mathcal{F}^{k-1}\la u_n+\alpha^2Au_n,(P_n Q)\dd W(s)\ra.
  \end{multline}
Before  taking expectation, we need to verify that the martingale term is integrable.
Arguing as before, 
there exists $c_4,c_5>0$ such that
\[
 \mathcal{F}^{k-1}|(P_n Q)^*(I+\alpha^2A) u_n|_2\leq  c_4\mathcal{F}^{k-1}(|u_n|_2+\alpha^2|\nabla u_n|_2)\leq c_5(1+\mathcal{F}^{k}).
\]
Then, since $\mathcal{F}^k(t\wedge \tau_n)\leq N^k$, we can take expectation  to obtain
\[ 
    2k\E\left[\int_0^{t \wedge \tau_N^n}\mathcal{F}^{k-1}\la u_n+\alpha^2Au_n,(P_n Q)\dd W(s)\ra\right]=0.
 \]
Finally, by taking expectation in \eqref{eq.boundsF} we get 
\begin{multline*} 
  \E[\mathcal{F}^k(t \wedge \tau_N^n)] +k\nu\E\left[\int_0^{t \wedge \tau_N^n}\mathcal{F}^{k-1} \times  \left(|\nabla u_n|_2^2       +\alpha^2|A  u_n|_2^2 
     \right) \dd s \right]
   \leq \E\left[\mathcal{F}^k(0)    \right] 
   +c_3\E\left[t \wedge \tau_N^n\right].                           
\end{multline*}
Letting $N\to\infty$ we conclude the proof.
\end{proof}

We can remark that the constant $c$ appearing in Theorem \ref{thm.moments} actually depends on $k$,  $\qspace$ (since we need a Poincaré inequality) and $\Tr[Q^*(I+\alpha^2A)Q]$. 

\begin{Coro} \label{coro.thight}
Let $u_0\in \sV$ and assume that Hypothesis \ref{hyp.trace} holds.
 Then 
  for any $n\in \N^*$
 \begin{multline*}
  \E\left[ |u_n(t)|_2^2 +\alpha^2 |\nabla u_n(t)|_2^2 \right]+2\nu\E\left[\int_0^{t}\left(|\nabla u_n|_2^2       +\alpha^2|A  u_n|_2^2 \right) \dd s\right]\\
     \leq \E\left[|u_0|_2^2 +\alpha^2 |\nabla u_0|_2^2\right]+\Tr[ Q^*(I+\alpha^2A)Q]t.
 \end{multline*}
\end{Coro}
\begin{proof}
 The proof is the same as done for Theorem \ref{thm.moments}. Indeed, in the case $k=1$, the Itô formula gives immediately the result.
\end{proof}
\begin{Th}   \label{thm.momentsup}
 Let $u_0\in \sV$ and assume that Hypothesis \ref{hyp.trace} holds.
 Then for any $k\in \N$ there exists $c=c(k,\qspace,Q)>0$ such that for any $T>0$
 \[
  \E\left[\sup_{t\in[0,T]}\left( |u_n|_2^2 +\alpha^2 |\nabla u_n|_2^2\right)^k \right] \leq c\E\left[\left( |u_0|_2^2 +\alpha^2 |\nabla u_0|_2^2\right)^k\right]+cT.
 \]
\end{Th}
\begin{proof}
As done for the previous Theorem, let us set $\mathcal{F}$ as in \eqref{eq.defF}.
By Theorem \ref{thm.moments} the solution $u_n$ is global and all moments of $\mathcal{F}$ have finite expectation.

Using Itô formula and arguing as for the proof of Theorem \ref{thm.moments}
we get (see formula \eqref{eq.boundsF}) 
 \begin{multline*} 
    \E\left[\sup_{t\in[0,T]}\mathcal{F}^k(t ) + k\nu\int_0^t\mathcal{F}^{k-1} \times  \left(|\nabla u_n|_2^2       +\alpha^2|A  u_n|_2^2  \right) \dd s\right] \\
    \leq \E\left[\mathcal{F}^k(0)\right] 
					   +c_1T +
 2k \E\left[\sup_{t\in[0,T]}\int_0^{t}\mathcal{F}^{k-1}\la u_n+\alpha^2Au_n,(P_n Q)\dd W(s)\ra\right],
  \end{multline*}
where $c_1>0$ depends by $k,Q,\qspace$.
For the martingale part, we can use Burkholder-Davis-Gundy inequality and \eqref{eq.QAbounded} to get  for some constants $c_2,c_3>0$, depending only on $k, Q$
\begin{eqnarray*}
  \E\left[\sup_{t\in[0,T]}\left|\int_0^{t}\mathcal{F}^{k-1}\la u_n+\alpha^2Au_n,(P_n Q)\dd W(s)\ra\right|\right] 
  &\leq& c_2\E\left[\int_0^{T} \mathcal{F}^{2(k-1)}\left|(P_n Q)^*(I+\alpha^2A)u_n\right|_2^2\dd s \right]^\frac{1}{2} \\
  &\leq&  c_3\E\left[\int_0^{T} \mathcal{F}^{2k}\dd s \right]^\frac{1}{2} <\infty
\end{eqnarray*}
Using Theorem \ref{thm.moments} and Young's inequality, we deduce that the last term is bounded by $c(\E\left[\mathcal{F}^k(0)\right]+T)$, where $c=c(k,Q,\qspace)$ is a positive constant.
This completes the proof.
\end{proof}

\section{Exponential moments}\label{sec:expo} 
The following result will be used to derive some concentration properties for the invariant measure.
\begin{Prop} \label{prop.momexp}
 Let $u_0\in \sV$ and assume that Hypothesis \ref{hyp.trace} holds.
 For any $\varepsilon>0$ such that $-\nu+2\lambda_1^{-1}\varepsilon\Tr[Q^*(I+\alpha^2A)Q]< 0$ there exists $ K_\varepsilon>0$, independent by $u_0$,  such that
 for any $n\in\N$
 \begin{multline*}
  \E\left[e^{\varepsilon\left(|u_n(t)|_2^2+\alpha^2|\nabla u_n(t)|_2^2\right)} \right]\\
  +\varepsilon \left(\nu-2\lambda_1^{-1}\varepsilon \Tr[Q^*(I+\alpha^2A)Q]\right)
  \int_0^{t}\E\left[e^{\varepsilon\left(|u_n(s)|_2^2+\alpha^2|\nabla u_n(s)|_2^2\right)}
  (|\nabla u_n(s)|_2^2+\alpha^2|Au_n(s)|_2^2) \right]\dd s\\
  \leq \E\left[e^{\varepsilon(|u_0|_2^2+\alpha^2|\nabla u_0|_2^2)}\right] +K_\varepsilon t.
 \end{multline*}
\end{Prop}

\begin{proof}
 Let $f_\varepsilon(x)=e^{\varepsilon(|x|_2^2+\alpha^2|\nabla x|_2^2)}$ and set 
 $\tau_n^N$ as in \eqref{eq.stopping}. 
 In the next calculus we shall use the notation $C_Q=\Tr[Q^*(I+\alpha^2A)Q]$. 
 By Itô formula and \eqref{eq.QAbounded} we have
 \begin{multline} \label{eq.proofmomexp}
  f_\varepsilon(u_n(t\wedge \tau_n^N))  
  c= f_\varepsilon(x)+2\varepsilon\int_0^{t\wedge \tau_n^N} f_\varepsilon(u_n)\la (I+\alpha^2A)u_n,(P_nQ)\dd W(s)\ra\\
  +2\varepsilon \int_0^{t\wedge \tau_n^N}f_\varepsilon(u_n)\left(-\nu(|\nabla u_n|_2^2+\alpha^2|Au_n|_2^2) 
 +2\varepsilon|Q^*(I+\alpha^2A)u_n|_2^2 +\frac12C_Q\right)\dd s\\
  \leq f_\varepsilon(x)+2\varepsilon \int_0^{t\wedge \tau_n^N}f_\varepsilon(u_n)\la (I+\alpha^2A)u_n,(P_nQ)\dd W(s)\ra
 \\
   +2\varepsilon \int_0^{t\wedge \tau_n^N}f_\varepsilon(u_n)\left(-\nu(|\nabla u_n|_2^2+\alpha^2|Au_n|_2^2) 
    +2\varepsilon C_Q(|u_n|_2^2+\alpha^2|\nabla u_n|_2^2) + \frac12 C_Q\right)\dd s 
\end{multline}
Taking into account that since $u_n\in \sV$ we have $\sqrt{\lambda_1}|u_n|_2\leq |\nabla u_n|_2$ and  $\sqrt{\lambda_1}|\nabla u_n|_2\leq |A u_n|_2$ the last term is bounded by
\begin{eqnarray*}
  f_\varepsilon(u_n(t\wedge \tau_n^N)) 
      & \leq & f_\varepsilon(u_0)+2\varepsilon\int_0^{t\wedge \tau_n^N} f_\varepsilon(u_n)\la (I+\alpha^2A)u_n,(P_nQ)\dd W(s)\ra\\
  &&+ 2\varepsilon \int_0^{t\wedge \tau_n^N} f_\varepsilon(u_n)\left(\left(-\nu+2\lambda_1^{-1}\varepsilon C_Q \right)(|\nabla u_n|_2^2+\alpha^2|A u_n|_2^2) 
     + \frac12 C_Q\right)\dd s
\end{eqnarray*}
Taking into account the obvious inequality $(-ax+b)e^x\leq -\frac{a}{2}xe^x+c$ for $x\geq0$ and some $c>0$ we deduce that there exists $K_\varepsilon>0$, independent by $u_n$, such that
\begin{eqnarray*}
  f_\varepsilon(u_n(t\wedge \tau_n^N)) & \leq &   f_\varepsilon(u_0)+2\varepsilon\int_0^{t\wedge \tau_n^N} f_\varepsilon(u_n)\la (I+\alpha^2A)u_n,(P_nQ)\dd W(s)\ra\\
 &&+ \int_0^{t\wedge \tau_n^N} \left( \varepsilon\left(-\nu+2\lambda_1^{-1}\varepsilon C_Q \right)(|\nabla u_n|_2^2+\alpha^2|A u_n|_2^2) f_\varepsilon(u_n)+K_\varepsilon\right)\dd s
\end{eqnarray*}
By taking expectation we obtain
\[
  \E[f_\varepsilon(u_n(t\wedge \tau_n^N))]+   \varepsilon \left(\nu-2\lambda^{-1}\varepsilon C_Q \right)\E\left[\int_0^{t\wedge \tau_n^N} f_\varepsilon(u_n(s))(|\nabla u_n|_2^2+\alpha^2|A u_n|_2^2)\dd s\right]
    \leq \E\left[f_\varepsilon(u_0)\right]+K_\varepsilon t.
\]
Notice that the martingale term is $\Prb$-integrable since by \eqref{eq.QAbounded}
\begin{eqnarray*}
 | f_\varepsilon(u_n) (P_nQ)^*(I+\alpha^2A)u_n|_2^2 &=& f_\varepsilon(\sqrt{2}u_n) |(P_nQ)^*(I+\alpha^2A)u_n|_2^2 \\
      & \leq & \Tr[Q^*(I+\alpha^2A)Q] f_\varepsilon(\sqrt{2}u_n)(|u_n|_2^2+\alpha^2|\nabla u_n|_2^2)
\end{eqnarray*}
and consequently, by the definition of the stopping time, we have
\[
   \sup_{s\in[0,t\wedge \tau_n^N[} \left\{| f_\varepsilon(u_n(s)) (P_nQ)^*(I+\alpha^2A)u_n(s)|_2^2\right\}
      \leq \Tr[Q^*(I+\alpha^2A)Q] e^{2\varepsilon N}N.
\]
Letting $N\to \infty$ we get that $\Prb$-a.s. $\tau_n^N\to \infty$ and  we obtain 
the result.
\end{proof}

%
%

%
%
%
%
%
\section{Estimates in Sobolev spaces}\label{sec:Sobolev}

%
 Let $X$ be a Banach space with norm $\|\cdot\|_X$.
 For $p\geq1$, $\theta\in]0,1[$ we denote by $W^{\theta,p}([0,T];X)$ the classical Sobolev space of all functions $f\in \sL^p([0,T];X)$ such that
 \[
   \int_0^T\int_0^T\frac{\|f(t)-f(s)\|_X^p}{|t-s|^{1+\theta p}}\dd s\dd t <\infty,
 \]
 endowed with the norm
 \[
   \|f\|_{W^{\theta,p}([0,T];X)}=\left(\|f\|_{L^p([0,T];X)}^p+\int_0^T\int_0^T\frac{\|f(t)-f(s)\|_X^p}{|t-s|^{1+\theta p}}\dd s\dd t  \right)^{\frac1p}.
 \]
The proof of the following lemma is left to the reader.
\begin{Le} \label{Le.sobolev}
Let $X$ a  Banach space.
For any $T>0$, $\theta\in]0,1/2[$, $p\geq2$ there exists $c=c(\theta,p,T)$ such that for any $f\in \sL^2([0,T];X)$ it holds 
 \[
   \left\|\int_0^\cdot f(\tau)\dd \tau \right\|_{W^{\theta,p}([0,T];X)}\leq c(\theta,p,T) \|f\|_{L^2([0,T];X)}.
 \]
\end{Le}
\begin{Prop}   \label{prop.holderw}
For $u_0\in \sV$, $T>0$, $n\in \N$, let $u_n$ be the solution of \eqref{eq.approx} in $[0,T]$.
  For any $T>0$, $\theta \in ]0,1/2[$,  $p\geq2$ there exists a constant $c=c(T,\theta,p)>0$ such that for any $n\in \N$ 
\[
   \E\left[ \|u_n\|_{W^{\theta,p}([0,T];\sH)}^2\right]
   \leq c\left(1+\frac{1}{\alpha^2}\right)\left(\E\left[|u_0|_2^2+\alpha^2|\nabla u_0|_2^2 \right] +1\right)^2.
\]
\end{Prop}

\begin{proof}
For any $n\in \N $,  we have
 \begin{eqnarray*}
   u_n(t)    &=&   P_nu_0
  -\nu\int_0^t    Au_n(\tau)\dd \tau
   - \int_0^t (I+\alpha^2A)^{-1}\widetilde{B}_{n}(u_n  ,u_n +\alpha^2 A u_n)(\tau) \dd\tau  
   +(P_n Q) W(t)\\
      &=&  P_nu_0+ J_1(t)+J_2(t)+J_3(t).
 \end{eqnarray*}
We proceed as for Proposition \ref{prop.holderw} by estimating each term.
Clearly, $\E\left[ \|P_n u_0  \|_{W^{\theta,p}([0,T];\sH)}^2\right]\leq T^{\frac2p}|u_0  |_2^2 $. 
For $J_1$ we have, using Lemma \ref{Le.sobolev} and Theorem \ref{thm.moments} (with $k=1$), that there exists $c_1>0$ such that
\[
 \E\left[ \| J_1(\cdot)\|_{W^{\theta,p}([0,T];\sH)}^2\right]
   \leq c(\theta,p)\E\left[\int_0^T|Au_n(\tau)|_2^2\dd \tau\right] 
     \leq \frac{c_1}{\alpha^2} \left(\E\left[|u_0|_2^2+\alpha^2|\nabla u_0|_2^2\right] +T \right)
\]
In order to estimate $J_2$, observe that by {\em (iv)} of Proposition \ref{prop.1.1} and Young inequality, for any $\xi \in \sH$ we have 
\begin{eqnarray*}
  \la  \widetilde{B} (P_nu_n ,P_n(u_n +\alpha^2 A u_n)),P_n(I+\alpha^2A)^{-1}\xi \ra_{(D(A)',D(A))}   &\leq & c  |u_n|_\sV\left(|u_n|_2 +\alpha^2 |A u_n|_2\right) |P_n(I+\alpha^2A)^{-1}\xi|_{D(A)}  \\ 
   &\leq&    \frac{c}{\alpha^2}  |u_n|_\sV\left(|u_n|_2 +\alpha^2 |A u_n|_2\right)|\xi|_2
\end{eqnarray*}
which implies
\begin{equation} \label{eq.ineqB}
  |(I+\alpha^2A)^{-1}\widetilde{B}_{n}(u_n  ,u_n +\alpha^2 A u_n)|_2\leq \frac{c}{\alpha^2}  |u_n|_\sV\left(|u_n|_2 +\alpha^2 |A u_n|_2\right).
\end{equation}
Since $|u_n|_\sV \leq c|\nabla u_n|_2$, by Young inequality $(a+b)^2\leq 2a^2+2b^2$ we get that for some $c>0$, independent by $u_n$, it holds
\[
  |(I+\alpha^2A)^{-1}\widetilde{B}_{n}(u_n  ,u_n +\alpha^2 A u_n)|_2^2
  \leq \frac{c}{\alpha^4}\left(|u_n|_2^2+\alpha^2|\nabla u_n|_2^2\right)\left(|\nabla u_n|_2^2 +\alpha^2 |A u_n|_2^2\right).
\]
By Lemma \ref{Le.sobolev} we deduce that there exists $c>0$ such that 
\begin{eqnarray*}
   \E\left[ \| J_2(\cdot)\|_{W^{\theta,p}([0,T];\sH)}^2\right]  &\leq& 
          c\E\left[\int_0^T|u_n|_\sV^2|(I+\alpha^2A)^{-1}\widetilde{B}_{n}(u_n  ,u_n +\alpha^2 A u_n)|_2^2\dd\tau \right] \\
         &\leq& \frac{c}{\alpha^4}\E\left[\int_0^T\left(|u_n|_2^2+\alpha^2|\nabla u_n|_2^2\right)\left(|\nabla u_n|_2^2 +\alpha^2 |A u_n|_2^2\right)\dd\tau \right] 
\end{eqnarray*}
Taking into account the bound given by Theorem \ref{thm.moments} we obtain that for some $c>0$ it holds
\[
  \E\left[ \| J_2(\cdot)\|_{W^{\theta,p}([0,T];\sH)}^2\right]
  \leq \frac{c}{\alpha^4}\left(\E\left[(|u_0|_2^2+\alpha^2|\nabla u_0|_2^2)^2\right]+T \right).
\]
For the last term, observe that by the gaussianity of $QW(t)$ that there exists $c=c(p)$ such that 
$\E[ |(P_n Q)W(t)|_2^p] \leq c (\Tr[Q^*Q])^{\frac{p}{2}} t^{\frac{p}{2}}$. 
Similarly, $\E[ |(P_n Q)(W(t)-W(s))|_2^p] \leq c (\Tr[Q^*Q])^{\frac{p}{2}} |t-s|^{\frac{p}{2}}$. 
Then, $\E[|J_3(\cdot)|^2_{L^p([0,T];\R)}]\leq c\Tr[Q^*Q]$ where $c>0$ depends only on $p,T$.
Since $2/p\leq1$ by Jensen inequality we get
\begin{eqnarray*} 
  \E\left[\left(\int_0^T\int_0^T\frac{|(P_n Q)(W(t)-W(s))|_2^p}{|t-s|^{1+\theta p}}\dd s\dd t\right)^{\frac{2}{p}}\right]
    &\leq&\left(\int_0^T\int_0^T\frac{\E[|(P_n Q)(W(t)-W(s))|_2^p]}{|t-s|^{1+\theta p}}\dd s\dd t\right)^{\frac{2}{p}}\\
    &\leq& \Tr[Q^*Q] \left(\int_0^T\int_0^T\frac{1}{|t-s|^{1+p(\theta-\frac12)}}\dd s\dd t\right)^{\frac{2}{p}}\\
    &\leq& c\Tr[Q^*Q]
\end{eqnarray*}
provided $\theta<1/2$.
Here, $c>0$ depends by $T,\theta, p$.
Taking into account the estimates on $P_nu_0, J_1,J_2,J_3$ we get the result.
 \end{proof}

%
%
%
%
%
\section{Compactness}\label{sec:compactness}

\begin{Le}[Tightness] \label{le.tightness}
 For $u_0\in \sV$, $T>0$, $n\in \N$, let $u_n$ the solution of \eqref{eq.approx} in $[0,T]$.
 Then, for any $p>2$, $\rho>0$, the laws of $u_n, n\in \N$ are tight in 
 \[
     \mathcal{C}([0,T];D(A^{-\rho}))\cap \sL^p([0,T];\sV).
 \]
 Moreover, the laws of $Au_n, n\in \N$ are tight in the space
 $
   \sL^2([0,T]; \sH)
 $
endowed with the weak topology.
\end{Le}
\begin{proof}
The classical interpolation inequality
\[
   \|u\|_{\sH^{1+\rho}}\leq  \|u\|_{\sH^{1}}^{1-\rho} \|u\|_{\sH^{2}}^\rho, \qquad \rho\in [0,1]
\]
implies
\[
     \|u\|_{\sH^{1+\frac{2}{p}}}^{p}\leq  \|u\|_{\sH^{1}}^{p-2} \|u\|_{\sH^{2}}^2, \qquad p\in [2,\infty[.
\]
Then, by Theorem \ref{thm.moments} and Proposition \ref{prop.holderw} we deduce  that $(u_n)_n$ is bounded in 
\begin{equation*} 
   \sL^p\left(\Omega;\sL^p([0,T];\sH^{1+\frac{2}{p}})\right)\cap \sL^2\left( \Omega; \sL^2([0,T];D(A))\right)\cap \sL^2\left( \Omega; W^{\theta, p}([0,T];\sH)\right)
\end{equation*}
for any $p\in]2,\infty[$ and $\theta<1/2$. 
Taking into account  \cite[Theorem 2.1 and Theorem 2.2]{FlandoliGatarek}, for any $p\in]2,\infty[ $ and $\theta<1/2$ such that $\theta p>1$ the embeddings 
\begin{eqnarray*}
  &&   W^{\theta, p}([0,T];\sH)  \hookrightarrow \mathcal{C}([0,T];D(A^{-\rho})),\qquad \rho>0\\
  &&  \sL^p([0,T];\sH^{1+\frac{2}{p}})     \cap   W^{\theta, p}([0,T];\sH)\hookrightarrow \sL^p([0,T];\sV)
\end{eqnarray*}
are compact. 
Moreover, 
we have that $\sL^2([0,T]; \sH)$ endowed with the weak topology is a complete metrizable space.
Then, by Theorem \ref{thm.moments}, we deduce that the laws of the random variables $Au_n$ are tight in  $\sL^2([0,T]; \sH)$, 
endowed with the weak topology.
 
Then, the result follows by Prokhorov's theorem. 
\end{proof}

\begin{Th} \label{th.compactness}
Let $u_0\in \sV$.
Then, there exists a probability space $(\tilde\Omega,\mathcal{\tilde F},\tilde\Prb)$, 
a cylindrical Wiener processes $\tilde W(t)$,  defined on $(\tilde\Omega,\mathcal{\tilde F},\tilde\Prb)$, a stochastic process
\begin{eqnarray*}
  && u\in    \mathcal{C}([0,T];D(A^{-\rho}))\cap \sL^p([0,T];\sV)\cap \sL^2([0,T];D(A)),\qquad \rho>0,
\end{eqnarray*}
adapted to the filtration generated by $\tilde W$ 
 and a subsequence ( for simplicity it is not relabeled ) such that for any  $p<\infty$  and $\tilde\Prb$-a.s. the solution $u_n$ of problem \eqref{eq.approx}
 with $\tilde W(t)$ instead of $W(t)$ satisfies
\begin{equation*}
 \begin{split}
 (i)\ & u_n\to u \quad \text{ strongly in } \mathcal{C}([0,T];D(A^{-\rho})),\, \rho>0\\ 
   (ii)\ & u_n\to u \quad \text{ strongly in } \sL^p([0,T]; \sV),\, p\in [1,\infty[\\
 (iii)\ & \int_0^T\la A u_n(s), \xi(s)\ra \dd s \to \int_0^T\la A u(s), \xi(s)\ra \dd s \text{ for any } \xi \in \sL^2([0,T]; \sH)\\
\end{split}
\end{equation*}
\end{Th}

\begin{proof}
 Taking into account Lemma \eqref{le.tightness}, by Skorohod representation theorem and by a diagonal extraction argument, there exists a probability space
$(\tilde\Omega,\mathcal{\tilde F},\tilde\Prb)$,  a cylindrical Wiener process $\tilde W(t)$  defined on $(\tilde\Omega,\mathcal{\tilde F},\tilde\Prb)$, 
a stochastic process $u$ such that the convergence conditions in  {\em (i)--(ii)} hold.
For {\em (iii)}, notice that there exists a random variable $v$ such that $Au_n\to v$ weakly in $\sL^2([0,T];\sH)$ (modulo a new subsequence). 
The fact that $v$ can be identify with $Au$ follows by the closure of the operator $A$ and by the density of $D(A)$ in $\sH$.
\end{proof}

\section{Existence and uniqueness - Proof ot Theorem \ref{thm.intro}}\label{sec:proof}

\subsection{Existence}

By Theorem \ref{th.compactness} we know that there exists a subsequence  
$(u_n)_n$, converging $\widetilde \Prb$-a.s. to a process $u\in  \mathcal{C}([0,T];D(A^{-\rho}))\cap \sL^p([0,T];\sV)\cap \sL^2([0,T];D(A))$, $\rho>0$, $p\geq1$. 

The rest of the proof will be splitted in several lemma :
in Lemma \ref{le.supbound}, we shall show that the process $u$ satisfies the bounds \eqref{eq.aNSCHestimate}.
Then,  in Lemma \ref{le.def.sol} we shall show that $u$ is a weak solution of  the abstract problem.
Finally, in Lemma \ref{le.uniqueness} we shall show that pathwise uniqueness holds, which will give the existence and uniqueness of a strong solution.

\begin{Le} \label{le.supbound}
 Under hypothesis of Theorem \ref{thm.intro}, we have that \eqref{eq.aNSCHestimate} holds.
 Moreover, $\tilde \Prb$-a.s. we have $u\in C([0,T];\sH_w)$, where $\sH_w$ is the space $\sH$ endowed with the weak topology. 
\end{Le}

\begin{proof}
Let us show the first bound of \eqref{eq.aNSCHestimate}. 
Let us notice that by the definition of the norm in $D(A^{\rho})$ it holds $\|u\|_{D(A^{-\rho})}\leq |u|_{\sH}$, for all $\rho>0$.
By Theorem \ref{th.compactness}, 
\[
   \sup_{t\in[0,T]}\|u(t)\|_{D(A^{-\rho})}= \lim_{n\to\infty}\left(\sup_{t\in[0,T]}\|u_n(t)\|_{D(A^{-\rho})}  \right)\leq \liminf_{n\to\infty}\left(\sup_{t\in[0,T]}|u_n(t)|_{\sH}  \right).
\]
By Fatou's lemma and Theorem \ref{thm.momentsup} we deduce that for any $k>0$ there exists $c>0$ depending on $k$ such that
\[
   \widetilde\E\left[  \sup_{t\in[0,T]}|u|_{\sH}^{2k}\right]
    \leq    \liminf_{n\to\infty}\widetilde\E\left[\sup_{t\in[0,T]}|u_n|_{\sH}^{2k}\right]  
    \leq c\left(\widetilde \E\left[(|u_0|_2^2+\alpha^2|\nabla u_0|_2^2)^k\right]+T \right)
\]
which implies that the first bound  in  \eqref{eq.aNSCHestimate} holds.
Let us show the second bound. 
Let us fix $p\geq1$.
Notice that by Theorem \ref{th.compactness} we have, $\widetilde \Prb$-a.s. that  for any $\xi \in \sL^2([0,T],\sH)$ the limit  $\xi(\cdot) |u_n(\cdot)|_{\sV}^{p/2}\to \xi(\cdot)|u(\cdot)|_{\sV}^{p/2}$ holds strongly in $L^2([0,T];\sH)$.
Then, since $Au_n\to Au$ holds weakly in $\sL^2([0,T];\sH)$ (see {\em (iii)} of Theorem \ref{th.compactness}),
\[
  \lim_{n\to\infty} \int_0^T \la Au_n(t),|u_n(t)|_{\sV}^{p/2}\xi(t)\ra \dd t = \int_0^T \la Au(t),|u(t)|_{\sV}^{p/2}\xi(t)\ra \dd t.
\]
We deduce that $\widetilde \Prb$-a.s. the limit $Au_n(\cdot)|u_n(\cdot)|_{\sV}^{p/2}\to  Au(\cdot) |u(\cdot)|_{\sV}^{p/2}$ holds weakly in $\sL^2([0,T];\sH)$.
By well known properties of weak limits and Fatou Lemma we get, using Theorem \ref{thm.momentsup}
\begin{eqnarray*}
   \widetilde\E\left[ \int_0^T|Au(t)|_2^2|u(t)|_{\sV}^p\ddt \right] 
   &\leq&  \widetilde\E\left[ \liminf_{n\to\infty }\int_0^T|Au_n(t)|_2^2|u_n(t)|_{\sV}^p\ddt \right]\\
   &\leq&  \liminf_{n\to\infty }\widetilde\E\left[ \int_0^T|Au_n(t)|_2^2|u_n(t)|_{\sV}^p\ddt \right]\\
   &\leq&  c\left( \widetilde \E\left[(|u_0|_2^2 +\alpha^2 |\nabla u_0|_2^2)^{p+1}\right]+T\right).
\end{eqnarray*}
which implies the second bound of \eqref{eq.aNSCHestimate}.

In order to complete the proof we need to show that $u\in C([0,T];\sH_w)$.
Observe that $\tilde \Prb$-a.s. we have $u\in \sL^{\infty}([0,T];\sH)\cap C([0,T];D(A^{-\rho}))$.
Then, $\Prb$-a.s., $u(t)\in \sH$ for any $t\in [0,T]$ and $u\in C([0,T];\sH_w)$ (see, for instance, \cite[page 263]{Temam}).
\end{proof}

\begin{Le} \label{le.def.sol}
Assume that the linear operator $Q$ satisfies Hypothesis \ref{hyp.trace}.
Let $\mu_{0}$ be a probability measure on $(\sH,\mathcal{B}(\sH))$ such that for some $k\geq1$ it holds
\[
    \int_\sH\left(|x|_2^2+|\nabla x|_2^2\right)^k\mu_{0}(dx)<\infty. 
\]
Then, there exists a weak solution of \eqref{eq.aNSabstract} in the sense of Definition \ref{def.sol.weak}. 
 \end{Le}
\begin{proof}
Let us first show that $u$ solve \eqref{eq.aNSabstract}.
Since $u_n$ solves \eqref{eq.approx}, it is sufficient to show that the right-hand side of \eqref{eq.approx} converges to the right-hand side of \eqref{eq.aNSabstract}.

Let $\xi \in \sL^2([0,T]; D(A)) $.  
By Theorem \ref{th.compactness},  {\em (iii)} we have 
\[
  \lim_{n\to\infty}\int_0^T\la   u_{n},\xi(t)\ra \dd t= \int_0^T\la   u ,\xi(t)\ra \dd t 
\]
and, by the Fubini theorem and the dominated convergence theorem,
\begin{eqnarray*}
\lim_{n\to\infty }\nu  \int_{0}^{T} \la\int_{0}^{t} Au_{n}(\tau)\dd\tau, \xi(t)\ra \ddt
   &=& \nu \int_{0}^{T}\left(\lim_{n\to\infty }\int_{0}^{t}\langle  Au(\tau), \xi(t) \rangle d\tau \right)\ddt \\
   &=& \nu  \int_{0}^{T} \la\int_{0}^{t}Au(\tau),\xi(t) \ddt  \ra \dd\tau.
\end{eqnarray*}
Observe that by Proposition \ref{prop.1.1} (ii) it holds, for some $c>0$,
\begin{multline*}
   \left|\int_0^T\int_0^t\la \widetilde{B}(u(\tau),v(\tau)),(I+\alpha^2A)^{-1}\xi(t)\ra \dd\tau\dd t\right| \\
     \leq c\left(\int_0^T|u(\tau)|_{\sV}^2\dd \tau \right)^{\frac12}\left(\int_0^T|v(\tau)|_{2}^2\dd \tau \right)^{\frac12}\left(\int_0^T|A(I+\alpha^2A)^{-1}\xi(t)|_{\sH}^2\dd t \right)^{\frac12}\\
     \leq \frac{c}{\alpha^2}\left(\int_0^T|u(\tau)|_{\sV}^2\dd \tau \right)^{\frac12}\left(\int_0^T|v(\tau)|_{2}^2\dd \tau \right)^{\frac12}\left(\int_0^T|\xi(t)|_{\sH}^2\dd t \right)^{\frac12}
\end{multline*}
This implies that the trilinear form 
\begin{eqnarray*}
 && \sL^2([0,T];\sV)\times  \sL^2([0,T];\sL^2(\qspace))\times  \sL^2([0,T];\sH) \to \R \\
 && (u,v,\xi)\mapsto \int_0^T \int_0^t\la (I+\alpha^2A)^{-1}\widetilde{B}(u(\tau),v(\tau)),\xi(t)\ra\dd\tau \dd t
\end{eqnarray*}
is continuous.
Since by Theorem \ref{th.compactness} we have that $ \widetilde \Prb$-a.s. $u_n\to u$ strongly in $\sL^2([0,T];\sV)$,
that $u_{n} + \alpha^{2}Au_{n}\to u + \alpha^{2}Au $ weakly in $\sL^2([0,T];\sL^2(\qspace))$ and clearly $P_n\xi\to \xi $ strongly in $\sL^2([0,T];\sH)$, 
we deduce that $\widetilde \Prb$-a.s.
\begin{eqnarray*}
  &&\lim_{n\to\infty } \int_{0}^{T} \la \int_0^t   (I+\alpha^2A)^{-1}\widetilde{B}_n(u_n  ,u_{n} + \alpha^{2}Au_{n})(\tau)\dd\tau,\xi(t)\ra\dd t
\\
 &&\qquad  = \lim_{n\to\infty } \int_{0}^{T} \int_0^t \la  (I+\alpha^2A)^{-1}\widetilde{B}(P_nu_n  ,P_n(u_{n} + \alpha^{2}Au_{n}))(\tau),P_n\xi(t)\ra  \dd\tau \dd t
\\
 &&\qquad =  \int_{0}^{T} \int_0^t \la (I+\alpha^2A)^{-1} \widetilde{B}(u , u + \alpha^{2}Au )(\tau),\xi(t)\ra  \dd\tau \dd t.
\end{eqnarray*}
Finally, it is easy to see that $\tilde \Prb$-a.s. it holds 
\[
   \lim_{n\to\infty} \int_0^T\la  \int_0^t(P_n Q)\dd\widetilde W(\tau), \xi(t)\ra \dd t 
   =  \int_0^T\la  \int_0^t  Q \dd\widetilde  W(\tau), \xi(t)\ra \dd t.   
\]
Let us show that $u$ as paths in $C([0,T];\sH)$.
Using Itô formula on  $|u_n|^2/2$ and integrating over $[0,T]$ we get 
\begin{multline*}
 \frac{1}{2}\int_0^T|u_n(t)|_2^2\dd t   =-\nu\int_0^T\int_0^t\left( |\nabla u_n(s)|_2^2\dd s\right)\ddt\\
   - \int_0^T\left(\int_0^t \la \widetilde{B}_n(u_n(s),u_n(s)+\alpha^2 A u_n(s)),(I+\alpha^2A)^{-1}u_n(s)\ra \dd s\right)\dd t +\int_0^T\left(\int_0^t \la u_n(s),P_nQ\dd W(s)\ra\right)\dd t.
\end{multline*}
By Theorem  \ref{th.compactness} and arguing as before all the terms of the previous formula converges and we get
\begin{multline*}
 \frac{1}{2}\int_0^T|u(t)|_2^2\dd t   =-\nu\int_0^T\int_0^t\left( |\nabla u(s)|_2^2\dd s\right)\ddt\\
   - \int_0^T\left(\int_0^t \la \widetilde{B}(u(s),u(s)+\alpha^2 A u(s)),(I+\alpha^2A)^{-1}u(s)\ra \dd s\right)\dd t +\int_0^T\left(\int_0^t \la u(s),Q\dd W(s)\ra\right)\dd t.
\end{multline*}
By identification, we have, $\ddt\times \widetilde \Prb$ a.e.
 \begin{multline*}
  \frac{1}{2}|u(t)|_2^2   =-\nu\int_0^t |\nabla u(s)|_2^2\dd s\\
    - \int_0^t \la \widetilde{B}(u(s),u(s)+\alpha^2 A u(s)),(I+\alpha^2A)^{-1}u(s)\ra \dd s +\int_0^t \la u(s),Q\dd W(s)\ra.
 \end{multline*}
By the square integrability of $u$, the last term on the right-hand side is a square integrable continuous martingale.
The term with the nonlinear part is integrable on $[0,T]$, since
\[
   \left| \la \widetilde{B}(u(s),u(s)+\alpha^2 A u(s)),(I+\alpha^2A)^{-1}u(s)\ra\right| \leq c|u(s)|_\sV (|u(s)|_2+\alpha^2 |A u(s)|_2)|A(I+\alpha^2A)^{-1}u(s)|_2 
\]
\[
  \leq \frac{c}{\alpha^2}|u(s)|_\sV (|u(s)|_2+\alpha^2 |A u(s)|_2)|u(s)|_2\leq  \frac{c}{\alpha^2}|u(s)|_\sV^2 (|u(s)|_2+\alpha^2 |A u(s)|_2)
\]
By \eqref{eq.aNSCHestimate}, the last term on the right-side belongs to $L^1([0,T];\R)$ $\widetilde \Prb$-almost surely.
We deduce that the map
\[
  t\mapsto -\nu\int_0^t |\nabla u(s)|_2^2\dd s-\int_0^t \la \widetilde{B}(u(s),u(s)+\alpha^2 A u(s)),(I+\alpha^2A)^{-1}u(s)\ra \dd s +\int_0^t \la u(s),QdW(s)\ra
\]
is $ \widetilde \Prb$-a.s. continuous.
Then $t\mapsto |u(t)|_2^2$ is  $ \widetilde \Prb$-a.s. continuous. 
Since by Lemma \ref{le.supbound} we know that $t\mapsto u(t)$ is weakly continuous in $\sH$, we deduce that $ \widetilde \Prb$-a.s. $u\in C([0,T];\sH)$, for all $T>0$.
The proof is complete. 
\end{proof}

\subsection{Uniqueness}
%
%
%

%

\begin{Le} \label{le.uniqueness}
 Under Hypothesis \ref{hyp.trace}, for any random variable  $u_0$ with values in $\sH$ and  such that 
 \[
    \E\left[ |u_0|_2^2+|\nabla u_0|_2^2 \right]<\infty
 \]
 there exists a unique strong solution of  problem \eqref{eq.aNSabstract} with initial value $u(0)=u_0$, in the sense of Definition \ref{def.sol.strong}.
\end{Le}

\begin{proof}
Let $\mu_{0}$ be the law of $u_0$ in $\sH$.
Notice that the hypothesis implies that $\mu_{0}$ is concentrated on $\sV$.
By Lemma \ref{le.supbound} and Lemma \ref{le.def.sol} we deduce that there exists a weak solution $(u,W)$
of problem \eqref{eq.aNSabstract} with initial distribution $\mu_{0}$ and such that the bounds 
\begin{equation} \label{eq.estimate.mu}
 \begin{split}
  &\E\left[\sup_{0\leq t\leq T} \left(|u(t)|_{2}^{2}\right)  \right] 
   \leq c\int_\sH\left(|x|_2^2  +\alpha^2|\nabla x|_2^2    \right)\mu_{0}(\dd x) ,\\
  &\E\left[\int_0^{T} 
        \left(|\nabla u|_2^2 +\alpha^2|A  u|_2^2 \right) \dd s\right]
   \leq c\int_\sH\left(|x|_2^2+\alpha^2|\nabla x|_2^2\right)\mu_{0}(\dd x)
   \end{split}
\end{equation}
hold.
By the Yamada-Watanabe theorem for SPDEs (see, for instance, \cite{RSZ08,Tappe13}), it is sufficient to show pathwise uniqueness of the solution.
Set $w=u(t,x)-v(t,x)$, where $(u,W), (v,W)$ are two solutions with same initial value $x\in \sV$.
Then, $w$ satisfies the equation
\[
      \frac{d}{dt}w +\nu Aw+(I+\alpha^2 A)^{-1}\left(\widetilde{B}(u,w+\alpha^2Aw)+\widetilde{B}(w,v+\alpha^2Av)\right)=0.
\]
We deduce 
\[
      \frac12 \frac{d}{dt} (|w|_2^2+\alpha^2|\nabla w|_2^2) =-\nu(|\nabla w|_2^2+\alpha^2|A w|_2^2) - \la \widetilde{B}(u,w+\alpha^2Aw),w\ra 
\]
since $\la w,\widetilde{B}(w,v+\alpha^2Av)\ra = 0$.
Notice that this last equality  needs    $w\widetilde{B}(w,v+\alpha^2Av)\in L^1([0,T]\times \qspace)$.
By  Proposition \ref{prop.1.1} and classic inequalities, we get
\begin{equation}\label{eq.Bzero}
\begin{split}
  \int_0^T \la w, \widetilde{B}(w,v+\alpha^2Av)\ra \dd t & \leq  c\int_0^T | w|_\sV|Aw|_2(|v|_2+\alpha^2|Av|_2)\dd t    \\
  &\leq   c\int_0^T\left(|w|_\sV^2|Aw|_2^2 + (|v|_2^2+\alpha^4|Av|_2^2)    \right) \dd t \\
  &\leq  c\sup_{t\in[0,T]}(|u|_\sV^2+|v|_\sV^2) \int_0^T(|Au|_2^2+|Av|_2^2)\dd t \\
  &\qquad + c\int_0^T\left (|v|_2^2+\alpha^4|Av|_2^2)    \right) \dd t  
\end{split}
\end{equation}
Here $c>0$ is some real constant which depends only on $T$ and $\qspace$ and can change line by line.
By Lemma \eqref{eq.estimate.mu} we have that if $u$, $v$ are solutions of \eqref{eq.aNSabstract} such that $u(0)=v(0)$ $\Prb$-a.s. 
and with initial distribution $\mu_{0}$, then there exists $c>0$ such that
\[
   \E\left[\sup_{t\in[0,T]}(|u|_\sV^2 +|v|_\sV^2)+\int_0^T(|Au|_2^2+|Av|_2^2)\dd t \right] \leq c\int_{\sH}(|x|_2^2 +\alpha^2|\nabla x|_2^2)\mu_{0}(\dd x)
\]
This implies that the last term in \eqref{eq.Bzero} is $\Prb$-a.s. finite and then $\la w,\widetilde{B}(w,v+\alpha^2Av)\ra $ is integrable and vanishes $\Prb$-a.s.

By Proposition \ref{prop.1.1} and Young inequality
\begin{eqnarray*}
  \left|\la \widetilde{B}(u,w+\alpha^2Aw),w\ra \right| &\leq& \left|\la \widetilde{B}(u,w),w\ra\right|+\alpha^2\left| \la \widetilde{B}(u,Aw),w\ra\right| \\
    &\leq & c|u|_V |w|_V^2+c\alpha^2|Au|_2|Aw|_2|w|_V \\
    &\leq & \frac{\nu}{2} |w|_\sV^2 + \frac{\nu\alpha^2}{2} |Aw|_2^2+ c\left(|u|_\sV^2+\alpha^2|Au|_2^2  \right) |w|_\sV^2 
\end{eqnarray*}
Then,
\[
     \frac{d}{dt} (|w|_2^2+\alpha^2|\nabla w|_2^2) +\nu(|\nabla w|_2^2+\alpha^2|A w|_2^2) \leq   c\left(|u|_\sV^2+\alpha^2|Au|_2^2  \right) |w|_\sV^2. 
\]
Since the quantity $ \int_0^T \left(|u(s)|_\sV^2+\alpha^2|Au(s)|_2^2  \right)ds$ is $\Prb$-a.s. bounded
(see \eqref{eq.estimate.mu}),
by Gronwall lemma we deduce that $\Prb$-a.s.
\[
  |w(t)|_2^2+\alpha^2|\nabla w(t)|_2^2\leq 0
\]
which implies $u(t,x)=v(t,x)$ for all $t\geq0$.
\end{proof}

%
%
%
%
%
\section{Invariant measure}\label{sec:invariant}
\subsection{A priori estimates}\label{subsec:derivative}

The aim of this section is to understand how the solutions $u_m$ of equation \eqref{eq.approx} depend on the initial data.
Then, we shall obtain suitable estimates on the derivative of the solution with respect to the starting point $x$. 
To do this,  the Gateaux derivative  $u_m$ with respect to the initial datum $x$ alongside the direction $h$ is denoted by $\eta_m^h(t,x) = Du_m(t,x)\cdot h$. 
It is well known that $\eta_m^h(t,x) $ is solution of the ordinary differential equation with random coefficients
\begin{multline}\label{eq.eta}
 \begin{cases}
    \frac{d}{dt} \eta_m^h(t,x) +\nu A\eta_m^h(t,x) + (I+\alpha^2A)^{-1}\widetilde{B}_m(\eta_m^h(t,x),u_m(t,x)+\alpha^2 Au_m(t,x))
   \\ \qquad+(I+\alpha^2A)^{-1}\widetilde{B}_m(u_m(t,x),\eta_m^h(t,x)+\alpha^2A\eta_m^h(t,x))=0\\
  \eta_m^h(0,x) = P_m h\hfill
   \end{cases}
\end{multline}


\begin{Prop} \label{prop.bound.AX}
Let $u_m(t),t\geq0$ be the solution of \eqref{eq.approx} starting by $x \in \sV$.
Then,  there exists $c>0$ depending on $Q$, $\qspace$ such that for any $t\geq0$, $\varepsilon > 0$
\\
\begin{multline*}
  \E\left[\exp\left(\varepsilon\left(\frac12|u_m(t)|_2^{2} + \frac{\alpha^{2}}{2}|\nabla u_m(t)|_2^{2}+\nu\int_0^t\left(|\nabla u_m(s)|_2^2+\alpha^2|Au_m(s)|_2^{2}\right)\dd s\right)\right)\right] \\
  \leq\exp\left(  \frac{\varepsilon}{2}\left(|x|_2^{2} + \alpha^{2}|\nabla x|_2^{2}\right)
  \left(1+\varepsilon\Tr[Q^*Q]t\right)\right)\\
 \times\exp\left(\frac{\varepsilon}{2}\Tr[Q^*(I+\alpha^2A)Q]
     \left(t +\frac{\varepsilon}{2}\Tr[Q^*(I+\alpha^2A)Q]t^2\right)\right)
\end{multline*}


\end{Prop}
\begin{proof}
By Itô formula we have
\begin{align*}
\frac12|u_m(t)|_2^{2} + \frac{\alpha^{2}}{2}|\nabla u_m(t)|_2^{2} &+  \nu\int_{0}^{t}(|\nabla u_m (s)|_2^{2} + \alpha^{2}|Au_m(s)|_2^{2})ds 
\\
&=\frac12|P_mx|_2^{2} + \frac{\alpha^{2}}{2}|\nabla P_mx|_2^{2} +\int_{0}^{t}\langle u_m, (Id+\alpha^{2} A)QdW_{s}\rangle\\
& +  \frac12\Tr[(P_mQ)^*(I+\alpha^2A)P_mQ]t
\end{align*}
and for $\varepsilon >0$
\begin{align*}
&\exp\left(\varepsilon\left(\frac12|u_m(t)|_2^{2} + \frac{\alpha^{2}}{2}|\nabla u_m(t)|_2^{2}+\nu\int_0^t\left(|\nabla u_m(s)|_2^2+\alpha^2|Au_m(s)|_2^{2}\right)\dd s\right)\right)\\
  & \qquad = \exp \left(\frac{\varepsilon}{2}\left( |P_mx|_2^{2} + \alpha^{2}|\nabla P_mx|_2^{2}\right) +\varepsilon \int_{0}^{t}\langle u_m, (Id+\alpha^{2} A)QdW_{s}\rangle+  \frac{\varepsilon}{2} \Tr[(P_mQ)^*(I+\alpha^2A)P_mQ]t\right)\\
  & \qquad\leq \exp\left( \frac{\varepsilon}{2}\left(|x|_2^{2} + \alpha^{2}|\nabla x|_2^{2}\right)
    + \varepsilon\int_{0}^{t}\langle u_m(s), (Id+\alpha^{2} A)QdW_{s}\rangle   +      \frac{\varepsilon}{2}\Tr[Q^*(I+\alpha^2A)Q]t\right)
\end{align*}
By taking expectation and taking into account \eqref{eq.QAbounded} we obtain
\begin{align*}
 &  \E\left[\exp\left(\varepsilon\left(\frac12|u_m(t)|_2^{2} + \frac{\alpha^{2}}{2}|\nabla u_m(t)|_2^{2}+\nu\int_0^t\left(|\nabla u_m(s)|_2^2+\alpha^2|Au_m(s)|_2^{2}\right)\dd s\right)\right)\right]\\
 & \qquad \leq  \exp\left(\frac{\varepsilon}{2}\left( |x|_2^{2} + \alpha^{2}|\nabla x|_2^{2}\right)+     { \frac{\varepsilon}{2}\Tr[Q^*(I+\alpha^2A)Q]t} \right)\\
 & \qquad \qquad   \times \E\left[\exp\left( \varepsilon\int_{0}^{t}\langle u_m(s), (I+\alpha^{2} A)QdW_{s}\rangle\right)\right]\\
 & \qquad =   \exp\left(\frac{\varepsilon}{2}\left( |x|_2^{2} + \alpha^{2}|\nabla x|_2^{2}\right)+     { \frac{\varepsilon}{2}\Tr[Q^*(I+\alpha^2A)Q]t} \right)\\
 & \qquad \qquad 
    \times\exp\left(\frac{\varepsilon^2}{2} \int_{0}^{t}\E[ |Q^*(I+\alpha^{2} A) u_m(s)|_2^{2}]ds\right)\\
    &  \qquad \leq \exp\left(\frac{\varepsilon}{2}\left( |x|_2^{2} + \alpha^{2}|\nabla x|_2^{2}\right)+     { \frac{\varepsilon}{2}\Tr[Q^*(I+\alpha^2A)Q]t} \right)\\
 & \qquad \qquad 
      \times\exp\left(  \frac{\varepsilon^2}{2}\Tr[Q^*(I+\alpha^2A)Q]\int_{0}^{t}\E\left[ |u_m(s)|_2^{2}+ \alpha^2|\nabla u_m(s)|_2^{2}\right]ds\right)
\end{align*}
By using  Corollary \ref{coro.thight} we deduce that the last term is bounded by
\begin{align*}
 &  \qquad \exp\left(\frac{\varepsilon}{2}\left( |x|_2^{2} + \alpha^{2}|\nabla x|_2^{2}\right)+     { \frac{\varepsilon}{2}\Tr[Q^*(I+\alpha^2A)Q]t} \right)\\
 & \qquad \qquad 
      \times\exp\left(   \frac{\varepsilon^2}{2}\Tr[Q^*(I+\alpha^2A)Q]\left(\int_0^t(|x|_2^{2} + \alpha^{2}|\nabla x|_2^{2}+\Tr[Q^*(I+\alpha^2A)Q]s)\dd s \right)\right)\\
 & \qquad =  \exp\left(\frac{\varepsilon}{2}\left( |x|_2^{2} + \alpha^{2}|\nabla x|_2^{2}\right)+      \frac{\varepsilon}{2}\Tr[Q^*(I+\alpha^2A)Q]t \right)\\
 & \qquad \qquad 
                \times\exp\left( \frac{\varepsilon^2}{2}\Tr[Q^*(I+\alpha^2A)Q] 
                   \left((|x|_2^{2} + \alpha^{2}|\nabla x|_2^{2})t+\Tr[Q^*(I+\alpha^2A)Q]\frac{t^2}{2}\right)\right)
 \end{align*}
which implies the desired result.
\end{proof}

\begin{Prop} \label{prop.Aetah}
 There exists $c>0$ and a continuous function $\Gamma:\R^2\to \R$ such that for any  $x\in \sV$, $h\in D(A)$, $m\in \N$, $t>0$ it holds
\begin{multline*}
 |\nabla\eta_m^h(t,x)|_2^2+\alpha^2|A\eta_m^h(t,x)|_2^2+\int_0^t e^{c\int_0^s|Au_m(\tau,x)|_2^2\dd \tau} \left(|A\eta_m^h(s,x)|_2^2+\alpha^2|A^{3/2}\eta_m^h(s,x)|_2^2 \right)\dd s \\
    \leq
     \Gamma(t,|x|_2^2+\alpha^2|\nabla x|_2^2) \left(|\nabla h|_2^2+\alpha^2 |A h|_2^2\right)
\end{multline*}
\end{Prop}
\begin{proof}
 By multilplying both sides of \eqref{eq.eta} by  $A(I+\alpha^2A)\eta_m^h(t,x)$ and integrating over $\Uset $ we get
\begin{eqnarray} 
  &&\frac12 \frac{d}{dt} \left(|\nabla\eta_m^h(t,x)|_2^2+\alpha^2|A\eta_m^h(t,x)|_2^2\right)+|A\eta_m^h(t,x)|_2^2+\alpha^2|A^{3/2}\eta_m^h(t,x)|_2^2= \notag \\
  &&\quad   =  \la \widetilde{B}_m(u_m(t,x),\eta_m^h(t,x)+\alpha^2A\eta_m^h(t,x))+\widetilde{B}_m(\eta_m^h(t,x),u_m(t,x),+\alpha^2A u_m(t,x)),A\eta_m^h(t,x) \ra\notag\\
  &&\quad    =\la \widetilde{B}_m(u_m(t,x),\eta_m^h(t,x),A\eta_m^h(t,x) \ra+\alpha^2\la \widetilde{B}_m(u_m(t,x),A\eta_m^h(t,x),A\eta_m^h(t,x) \ra\notag\\
  &&\qquad    + \la \widetilde{B}_m(\eta_m^h(t,x),u_m(t,x),A\eta_m^h(t,x) \ra      + \alpha^2\la \widetilde{B}_m(\eta_m^h(t,x),Au_m(t,x),A\eta_m^h(t,x) \ra   \label{eq.Aeta}
\end{eqnarray}
By Proposition \ref{prop.1.1} and Young inequality we get that for some $c>0$ 

\begin{multline*}
 \la \widetilde{B}_m(u_m(t,x),\eta_m^h(t,x)),A\eta_m^h(t,x) \ra \leq c|Au_m(t,x)|_2|\nabla\eta_m^h(t,x)|_2|A\eta_m^h(t,x)|_2 \\
     \leq c|Au_m(t,x)|_2^2|\nabla\eta_m^h(t,x)|_2^2+\frac12|A\eta_m^h(t,x)|_2^2;
\end{multline*}
\begin{multline*}
  \alpha^2\la \widetilde{B}_m(u_m(t,x),A\eta_m^h(t,x)),A\eta_m^h(t,x) \ra \leq c|Au_m(t,x)|_2|A\eta_m^h(t,x)|_2|A^{3/2}\eta_m^h(t,x)|_2 \\
     \leq c\alpha^2|Au_m(t,x)|_2^2|A\eta_m^h(t,x)|_2^2+\frac{\alpha^2}{2}|A^{3/2}\eta_m^h(t,x)|_2^2;
\end{multline*}
\begin{multline*}
 \la \widetilde{B}_m(\eta_m^h(t,x),u_m(t,x)),A\eta_m^h(t,x) \ra  \leq c|Au_m(t,x)|_2|\nabla\eta_m^h(t,x)|_2|A\eta_m^h(t,x)|_2 \\
     \leq c|Au_m(t,x)|_2^2|\nabla\eta_m^h(t,x)|_2^2+\frac12|A\eta_m^h(t,x)|_2^2;
\end{multline*}
\begin{multline*}
  \alpha^2\la \widetilde{B}_m(\eta_m^h(t,x),Au_m(t,x)),A\eta_m^h(t,x) \ra    \leq c\alpha^2 |Au_m(t,x)|_2|A\eta_m^h(t,x)|_2|A^{3/2}\eta_m^h(t,x)|_2 \\
     \leq c\alpha^2|Au_m(t,x)|_2^2|A\eta_m^h(t,x)|_2^2+\frac{\alpha^2}{2}|A^{3/2}\eta_m^h(t,x)|_2^2.
\end{multline*}
Then,  the right-hand side of \eqref{eq.Aeta} is bounded by
\[
   c|Au_m(t,x)|_2^2\left(|\nabla \eta_m^h(t,x)|_2^2+\alpha^2|A\eta_m^h(t,x)|_2^2\right)+\frac12|A\eta_m^h(t,x)|_2^2+\frac{\alpha^2}{2}|A^{3/2}\eta_m^h(t,x)|_2^2
\]
and we get 
\begin{multline*}
 \frac12 \frac{d}{dt} \left(|\nabla\eta_m^h(t,x)|_2^2+\alpha^2|A\eta_m^h(t,x)|_2^2\right)+\frac12|A\eta_m^h(t,x)|_2^2+\frac{\alpha^2}{2}|A^{3/2}\eta_m^h(t,x)|_2^2\\
   \leq c|Au_m(t,x)|_2^2\left(|\nabla \eta_m^h(t,x)|_2^2+\alpha^2|A\eta_m^h(t,x)|_2^2\right).
\end{multline*}
By Gronwall inequality we obtain 
\begin{multline*}
 |\nabla\eta_m^h(t,x)|_2^2+\alpha^2|A\eta_m^h(t,x)|_2^2+\int_0^t e^{c\int_0^s|Au_m(\tau,x)|_2^2\dd \tau}\left(|A\eta_m^h(s,x)|_2^2+\alpha^2|A^{3/2}\eta_m^h(s,x)|_2^2 \right)\dd s \\
    \leq e^{c\int_0^t|Au_m(s,x)|_2^2\dd s}\left(|\nabla h|_2^2+\alpha^2 |A h|_2^2\right).
\end{multline*}
The result follows by Proposition \ref{prop.bound.AX}.
%
\end{proof}

%
%
%
%
%
\subsection{Proof of Theorem \ref{thm.invmeas}}\label{subsec:invariant}


\begin{Rem}
Following a classic strategy (see  \cite{DaPratoZabczyk:92}), we shall show a  Strong Feller type property  be using  the Bismut-Elworthy formula.
Indeed, to apply it we need $\ker(Q)=\{0\}$ and that $\int_0^T|Q^{-1}\eta_m^h(t,x)|_2^2\dd t$ is defined.
 Since $\int_0^T|A^{3/2}\eta_m^h(t,x)|_2^2 \dd t$ is bounded, in order to have $\int_0^T|Q^{-1}\eta_m^h(t,x)|_2^2\dd t$ defined it is sufficient  to have $D(A^{3/2})\subset D(Q^{-1})$.
 If we set $Q=A^{-\frac{1}{2}(1+\varepsilon)}$ this condition  is fulfilled if $\frac{1}{2}(1+\varepsilon)\leq 3/2$, that is $\varepsilon\leq 2$.
  Moreover, Hypothesis \ref{hyp.trace} reads 
 \[
   \Tr[A^{-\varepsilon}]<\infty.
 \]
That is, $\varepsilon>d/2$, where $d$ is the dimension of $\qspace$.

Then, if $\dim \qspace =2$ or $\dim \qspace=3$, a covariance operator $Q$ of the form $Q=A^{-\frac{1}{2}(1+\varepsilon)}$ %
satisfies the conditions  \ref{hyp.trace} and $D(A^{3/2})\subset D(Q^{-1})$ whenever $\varepsilon\in ]d/2, 2]$.
\end{Rem}
Before giving the proof of Theorem \ref{thm.invmeas}, we need two lemma.

%
%
\begin{Le}
 Under the hypothesis of Theorem \ref{thm.invmeas} for any $\phi\in B_b(\sV;\R)$, $t>0$, $x_0\in \sV$, $r>0$ we have
\[
 \lim_{|h|_{D(A)}\to 0}\sup_{|x-x_0|_\sV<r}|P_t\varphi(x+h)-P_t\varphi(x)| = 0.
\]
 \end{Le}
%
%
\begin{proof}
Let us begin with a function $\phi\in C_b(\sV,\R)$. 
Then, we shall extend the results to Borel and bounded functions.
For any $m\in \N^*$ the Bismuth-Elworthy formula yields 
\[ 
DP_t^{m} \phi (x)\cdot h = 
  \frac1t \E\left[\phi(u_{m}(t,x)) \int_{0}^{t} \langle Q^{-1}\eta^{h}_{m}(s,x), \dd W_{s}\rangle_K\right]
\]
Using Proposition \ref{prop.Aetah} the last term is bounded by
%
\begin{eqnarray*}
|DP_t^{m}  \phi (x)\cdot h | 
&\leq&  \frac1t  \|\phi\|_{\infty} \E\left[\int_{0}^{t} |Q^{-1}\eta^{h}_{m}(s,x)|_K^{2} \dd s\right]
\\
&\leq& \frac1t  \|\phi\|_{\infty} C_Q\left(\E\left[\int_{0}^{t} |A^{3/2}\eta^{h}_{m}(s,x)|^{2} \dd s\right]\right)^{\frac12}\\
&\leq& \frac{1}{\alpha t}  \|\phi\|_{\infty} C_Q  \Gamma(t,|x|_2^{2} + \alpha^{2}|\nabla x|_2^{2})
     \left(|\nabla h|_2^2+\alpha^2 |A h|_2^2\right)^{\frac12}.
\end{eqnarray*}
where $\Gamma:\R^2\to \R$ is a suitable continuous function.
Here, $Q:K\to\sH$ and $C_Q=\|Q^{-1}A^{-\frac32}\|_{\mathcal{L}(\sH,K)}$.\\
Then, since $P_t^m\phi(x)\to P_t\phi(x)$ as $m\to\infty$, 
\begin{eqnarray*}
  |P_t\phi(x+h)-P_t\phi(x)|&=&\lim_{m\to\infty}|P_t^m\phi(x+h)-P_t^m\phi(x)|\\
  &\leq&\lim_{m\to\infty}\left| \int_0^1   DP_t^{m} \phi (x+\theta h)\cdot h \dd \theta\right|\\
  &\leq& \frac{1}{\alpha t}  \|\phi\|_{\infty} C_Q  
  \sup_{\theta\in[0,1]}  \left\{
      \Gamma\left(t, |x+\theta h|_2^{2} + \alpha^{2}|\nabla (x+\theta h)|_2^{2}\right)  \right\}  \left(|\nabla h|_2^2+\alpha^2 |A h|_2^2\right)^{\frac12} 
\end{eqnarray*}
By approximating a Borel and bounded function by continous functions, we deduce that the previous estimate holds also for $\phi$ Borel and bounded.
We deduce that for a fixed $t>0$,
\[
  |P_t\phi(x+h)-P_t\phi(x)|\to 0
\]
 uniformly in any bounded set of $\sV$, as $h\to 0$ in $D(A)$.
\end{proof}

The next result  concern the irriducibility of the semigroup $P_t$.

\begin{Le}
 Under the hypothesis of Theorem \ref{thm.invmeas}  for any $\delta>0$, $x\in \sV$ there exists $T>0$ such that $\Prb(|u(T,x)|_2>\delta)<1$. 
\end{Le}

\begin{proof}
 We set $v^{m}(t)=u^{m}(t,x)-P_{m}W_A(t)$, where $W_A$ is the solution of the linear stochastic equation 
\[
   \begin{cases}
       \dd Z=-AZ\dd t +Q\dd W(t)\\
       Z(0)=x.
   \end{cases}
\]
Then, $v^{m}$ solves
\[ 
  \begin{cases}
      \dd v^{m}(t)= \left(-Av^{m}(t) +(I+\alpha^2 A)^{-1}\widetilde B_{m}(v^{m}+P_{m}W_A,v^{m}+P_{m}W_A+\alpha^2 A (v^{m}+P_{m}W_A))\right)\dd t,\\
      v^{m}(0)=0.
  \end{cases}
\]
By multilplying by $(I+\alpha^2 A)v^{m}$ both sides and integrating over $\qspace $ we get
\begin{eqnarray*}
  \frac{1}{2}\frac{\dd}{\dd t}( |v^{m}(t)|_2^2+\alpha^2|\nabla v^{m}|_2^2 ) &=&
      -\nu \left( |\nabla v^{m}(t)|_2^2+\alpha^2|Av^{m}|_2^2  \right) \\
    && +\la \widetilde B_{m}(v^{m}+P_{m}W_A,v^{m}+P_{m}W_A+\alpha^2 A (v^{m}+P_{m}W_A)),v^{m} \ra\\
    &=& -\nu \left( |\nabla v^{m}(t)|_2^2+\alpha^2|Av^{m}|_2^2  \right) \\
    && +\la \widetilde B_{m}(P_{m}W_A,v^{m}+\alpha^2 A v^{m}),v^{m} \ra  \\ 
    &&+ \la \widetilde B_{m}(P_{m}W_A,P_{m}W_A+\alpha^2 A P_{m}W_A),v^{m} \ra
\end{eqnarray*}
By Proposition \ref{prop.1.1} we have
\begin{eqnarray*}
  \la \widetilde B_{m}(P_{m}W_A,v^{m}+\alpha^2 A v^{m}),v^{m} \ra     &\leq&     c |AW_A|_2|v^{m}|_2|v^{m}|_\sV + c\alpha^2|AW_A|_2|Av^{m}|_2|v^{m}|_\sV
\\
   &\leq&  c |AW_A|_2|v^{m}|_2|v^{m}|_\sV + c\alpha^2|AW_A|_2|Av^{m}|_2|v^{m}|_\sV.
\end{eqnarray*}
By the Poincaré inequality $|v^{m}|_\sV\leq c|\nabla v^{m}|_2$ the last term is bounded by
\[
 \leq  c |AW_A|_2|v^{m}|_2|\nabla v^{m}|_2 + c\alpha^2|AW_A|_2|Av^{m}|_2|\nabla v^{m}|_2.
\]
Using Young inequality, we can bound this last quantity to obtain
\begin{equation} \label{eq.irr1}
  \la \widetilde B_{m}(P_{m}W_A,v^{m}+\alpha^2 A v^{m}),v^{m} \ra \leq  
   c |AW_A|_2^2\left(| v^{m}|_2^2 + \alpha^2|\nabla v^{m}|_2^2\right) +\frac{1}{2}|\nabla v^{m}|_2^2+\frac{\alpha^2}{2}|Av^{m}|_2^2
\end{equation}
Still by Proposition \ref{prop.1.1} we have
\begin{eqnarray*}
  \left|\la \widetilde B_{m}(P_{m}W_A,P_{m}W_A+\alpha^2 AP_{m}W_A),v^{m} \ra\right|     &\leq&    
         \left|\la \widetilde B_{m}(P_{m}W_A,P_{m}W_A),v^{m} \ra\right|+\alpha^2\left|\la \widetilde B_{m}(P_{m}W_A,AP_{m}W_A),v^{m} \ra\right|
\\
   &\leq&  c |W_A|_\sV^{3/2}|AW_A|_2^{1/2}|v^{m}|_2 + c\alpha^2|AW_A|_2^2|v^{m}|_\sV\\
   &\leq&  c |AW_A|_2^2\left(|v^{m}|_2+\alpha^2|\nabla v^{m}|_2  \right).
\end{eqnarray*}
In the last inequality we used the fact that $|z|_\sV\leq c\|z\|_{\sH^2}\leq c|Az|_2$ for some $c>0$, independent by $z\in D(A)$.
Using the inequality $a\leq 2+2a^2$ we get for some $c>0$  
\begin{equation}   \label{eq.irr2}
  \left|\la \widetilde B_{m}(P_{m}W_A,P_{m}W_A+\alpha^2 AP_{m}W_A),v^{m} \ra\right|    \leq  c |AW_A|_2^2\left(1+|v^{m}|_2^2+\alpha^2|\nabla v^{m}|_2^2  \right).
\end{equation}
Taking into account \eqref{eq.irr1} and \eqref{eq.irr2} we obtain 
\[
  \frac{\dd}{\dd t}( |v^{m}(t)|_2^2+\alpha^2|\nabla v^{m}|_2^2 )+ \frac{1}{2}|\nabla v^{m}|_2^2+\frac{\alpha^2}{2}|Av^{m}|_2^2
    \leq c |AW_A|_2^2\left(1+|v^{m}|_2^2+\alpha^2|\nabla v^{m}|_2^2  \right)
\]
Then, by Gronwall lemma we get, for some $c>0$ independent by $m$ and $v^{m}$
\begin{eqnarray*}
  |v^{m}(t)|_2^2+\alpha^2|\nabla v^{m}|_2^2&+&\int_0^te^{c\int_0^s|AW_A(\tau)|_2^2\dd \tau}\left(|\nabla v^{m}|_2^2+\alpha^2|Av^{m}|_2^2 \right)\dd s\\
     &\leq& c\int_0^te^{c\int_0^s|AW_A(\tau)|_2^2\dd \tau}|AW_A(s)|_2^2\dd s \\
     &=& c\left(e^{c\int_0^t|AW_A(s)|_2^2\dd s}-1 \right).
\end{eqnarray*}
Then, we deduce
\begin{eqnarray}
  \Prb(\{|u^{m}(T,x)|_2>\delta\})&\leq& \Prb(\{|v^{m}(T,x)|_2^2+|W_A(T)|_2^2>\delta^2/2\})\\
    &\leq&   \Prb(\{c'(e^{c\int_0^T|AW_A(s)|_2^2\dd s}-1)+|W_A(T)|_2^2>\delta^2/2\})
\end{eqnarray}
where $c,c'>0$ are real constants, independent by $m$.
By the gaussianity of $W_A$ we deduce $\Prb(\{|u^{m}(T,x)|_2>\delta\})<1-\varepsilon$, where $\varepsilon >0$ is independent by $m$.
 \end{proof}

%

\begin{proof}[Proof of Theorem \ref{thm.invmeas}]
By Corollary \ref{coro.thight} and Krylov-Bogolioubov theorem we deduce that there exists an invariant probability measure $\mu$ for the transition semigroup $P_t$, $t\geq0$.

%


{\bf Claim : $\mu(D(A))=1$ and \eqref{eq.muexp} holds}.\\

We consider, for $M>0$ and $\varepsilon$ such that $-\nu+2\varepsilon\lambda_1^{-1}\Tr[Q^*(I+\alpha^2A)Q]< 0$ the function 
\begin{eqnarray*}
\varphi_{\varepsilon,M}(y)=  \begin{cases} 
                   \frac{Me^{\varepsilon(|x|_2^2+\alpha^2|\nabla x|_2^2)}(|\nabla x|_2^2+\alpha^2|A x|_2^2)}{M+e^{\varepsilon(|x|_2^2+\alpha^2|\nabla x|_2^2)}(|\nabla x|_2^2+\alpha^2|A x|_2^2)} & \text{if }y\in D(A)\\
                   M  &\text{elsewhere}
                 \end{cases}
\end{eqnarray*}
which is continuous and bounded in $\sV$.
Let $x_0\in \sV$ where the ergodic theorem applies for $\varphi_{\varepsilon,M}$ :
\[
\int_{\sV}\varphi_{\varepsilon,M}(x)\mu(dx) = \lim_{T\to\infty}\frac{1}{T}\int_0^T P_t\varphi_{\varepsilon,M}(x_0)\dd t.
\]
Since $\varphi_{\varepsilon,M}$ is bounded and continuous on $\sV$, $P_t^m \varphi_{\varepsilon,M}(x_0)\to P_t \varphi_{\varepsilon,M}(x_0)$ as $m\to \infty$.
Then by Proposition \ref{prop.momexp} 
\begin{eqnarray*}
   \frac{1}{T}\int_0^T P_t\varphi_{\varepsilon,M}(x_0)\dd t 
       &=& \frac{1}{T}\lim_{m\to \infty}  \int_0^T P_t^m\varphi_{\varepsilon,M}(x_0)\dd t\\
    &\leq& \frac{1}{T}\lim_{m\to \infty}  \int_0^T P_t^m
    \left(e^{\varepsilon(|\cdot|_2^2+\alpha^2|\nabla \cdot|_2^2)}(|\nabla\cdot|_2^2+\alpha^2|A \cdot|_2^2)   \right)(x_0)\dd t\\
       &\leq & \frac{1}{T} \frac{e^{\varepsilon(|x_0|_2^2+\alpha^2|\nabla x_0|_2^2)} +K_\varepsilon T}{\varepsilon(\nu-2\varepsilon\lambda_1^{-1}\Tr[Q^*(I+\alpha^2A)Q])}.
\end{eqnarray*}
By letting $T\to\infty$ and we get that for any $M>0$ 
\[
   \int_{\sV}\varphi_{\varepsilon,M}(x)\mu(dx) 
   \leq  \frac{K_\varepsilon }{\varepsilon(\nu-2\varepsilon\lambda_1^{-1}\Tr[Q^*(I+\alpha^2A)Q])}.
\]
Then by Fatou lemma we obtain \eqref{eq.muexp} which implies $\mu(D(A))=1$.


{\bf Claim : uniqueness }\\
Let us assume that $\mu$ is an invariante measure and $x_0\in D(A)$ is in the support of $\mu$.
We shall show that for any $\delta>0$, $\mu(B_{\sV}(0,\delta))>0$. 
Let us fix $\delta>0$.
Since the semigroup is irreductible in $\sV$,  there exists $t>0$, $r>0$ such that $P_t\chi_{B_{\sV}(0,\delta)}(x_0)\geq r$.
By the strong Feller property, $x\mapsto P_t\chi_{B_{\sV}(0,\delta)}(x)$ is continuous in $D(A) $ and $P_t\chi_{B(0,\delta)}(x)\geq r/2$ is some ball $B_{D(A)}(x_0,\varepsilon)$, where $\varepsilon>0$. 
This implies 
\begin{eqnarray*}
  \mu(B_{\sV}(0,\delta'))
   &=& \int_\sV P_t\chi_{B_{\sV}(0,\delta)}(x)\mu(dx) \\
   &\geq&    \int_{B_{D(A)}(x_0,\varepsilon)} P_t\chi_{B_{\sV}(0,\delta)}(x)\mu(dx)\\
   &\geq& \frac{r}{2}\mu(B_{D(A)}(x_0,\varepsilon))>0.
\end{eqnarray*}
Since $x=0$ is in the support of any invariant measure, we deduce that the invariant measure is unique.
\end{proof}
\end{document}